\documentclass[leqno,11pt]{amsart}%
\usepackage{eurosym}
\usepackage{amsmath}
\usepackage{mathrsfs}
\usepackage{amsfonts}
\usepackage{graphicx}
\usepackage{color}
\usepackage{amsfonts}
\usepackage{amssymb}%
\usepackage{esint}
\usepackage{todonotes}
\usepackage[abbrev]{amsrefs}
\usepackage{mathtools}
\usepackage{tikz-cd}
\usepackage{float}

\allowdisplaybreaks

\usepackage{amsfonts}
\usepackage{amsxtra}

\numberwithin{equation}{section}
  \providecommand{\Xint}[1]{\mathchoice
    {\XXint\displaystyle\textstyle{#1}}%
    {\XXint\textstyle\scriptstyle{#1}}%
    {\XXint\scriptstyle\scriptscriptstyle{#1}}%
    {\XXint\scriptscriptstyle\scriptscriptstyle{#1}}%
    \!\int}
  \providecommand{\XXint}[3]{{\setbox0=\hbox{$#1{#2#3}{\int}$}
      \vcenter{\hbox{$#2#3$}}\kern-.5\wd0}}
  
  \providecommand{\dashint}{\mathop{\Xint-}}

\newtheorem{theorem}{Theorem}[section]

\newtheorem{corollary}[theorem]{Corollary}

\newtheorem{example}[theorem]{Example}

\newtheorem{lemma}[theorem]{Lemma}

\newtheorem{remark}[theorem]{Remark}

\newtheorem{theoremalph}{Theorem}

\newtheorem{definitionalph}{Definition}

\numberwithin{equation}{section}
\DeclareFontFamily{U}{mathx}{}
\DeclareFontShape{U}{mathx}{m}{n}{<-> mathx10}{}
\DeclareSymbolFont{mathx}{U}{mathx}{m}{n}
\DeclareMathAccent{\widehat}{0}{mathx}{"70}
\DeclareMathAccent{\widecheck}{0}{mathx}{"71}

\def\Mpl{\mathcal M_+}
\def\M{\mathcal M}

\def\om{\mathbb{R}^N}
\def\rn{\mathbb{R}^n}
\def\rM{\mathbb{R}^m}

\def\M{{\mathcal M}}
\def\N{{\mathbb N}}

\newcommand{\smallD}{{\scriptscriptstyle D}}

\def\supp{{\rm supp\,}}

\newcommand{\R}{\mathbb{R}}

\newcommand{\barint}{
\rule[.036in]{.12in}{.009in}\kern-.16in \displaystyle\int }

\newcommand{\barcal}{\mbox{$ \rule[.036in]{.11in}{.007in}\kern-.128in\int $}}

\usepackage{enumitem}
\makeatletter
\let\@wraptoccontribs\wraptoccontribs
\makeatother

\mathchardef\mhyphen="2D

\title[
  Riesz potential estimates under co-canceling constraints
 ]{
  Riesz potential estimates under co-canceling constraints
}

\author {Dominic Breit, Andrea Cianchi \& Daniel Spector}

\address{Dominic Breit, 
  Institute of Mathematics\\
  TU Clausthal, 
Erzstra\ss e 1, 
Clausthal-Zellerfeld, 
38678 Germany} \email{dominic.breit@tu-clausthal.de}

\address{Andrea Cianchi, Dipartimento di Matematica e Informatica \lq\lq U. Dini"\\
Universit\`a di Firenze,
Viale Morgagni 67/a,
50134 Firenze,
Italy} \email{andrea.cianchi@unifi.it}

\address{Daniel Spector, Department of Mathematics \\ National Taiwan Normal University  No. 88, Section 4, Tingzhou Road, Wenshan District, Taipei City, Taiwan 116, R.O.C.
\newline
and
\newline
National Center for Theoretical Sciences\\
No. 1 Sec. 4 Roosevelt Rd., National Taiwan
University\\
Taipei, 106, Taiwan
\newline
and
\newline
Department of Mathematics \\ 
University of Pittsburgh, 
Pittsburgh, PA 
15261 
USA
}
\email{spectda@gapps.ntnu.edu.tw}

\subjclass[2020]{42B99, 46E35, 46E30}
\keywords{Riesz potentials; Sobolev inequalities;  co-canceling differential operators;    rearrangement-invariant spaces; Orlicz spaces}

\begin{document}

\begin{abstract} 
Inequalities  for Riesz potentials are well-known to be equivalent to Sobolev inequalities of the same order for domain norms \lq\lq far" from $L^1$, but to be weaker otherwise. Recent contributions by Van Schaftingen,  by Hernandez, Rai\c{t}\u{a} and Spector, and by Stolyarov 
  proved that this gap can be filled in Riesz potential inequalities for vector-valued functions in  $L^1$ fulfilling a co-canceling differential condition. The present work demonstrates that such a property is not 
 just peculiar to the space $L^1$. 
 As   a consequence, Riesz potential inequalities under the  co-canceling constraint   are offered for general families of rearrangement-invariant spaces, such as the Orlicz spaces and the Lorentz-Zygmund spaces. Especially relevant instances of inequalities for domain spaces neighboring   $L^1$ are singled out.
\end{abstract}

\maketitle

%\todo[inline]{A: we have to add \lq\lq elliptic" whenever we deal with canceling operators. To avoid repeating \lq\lq $k$-th order homogeneous elliptic canceling differential operator" every time, shall we just define it a \lq\lq $k$-th order admissible operator"?}

\section{Introduction}\label{intro}

As was shown in the influential work of Sobolev \cite{sobolev}, the inequalities nowadays named after him are intimately connected to estimates for Riesz potentials in the Lebesgue spaces $L^p$, for $p>1$. On the other hand,  they take a different form in the borderline case when $p=1$.
The moral this paper aims at advertising is that this connection can be restored in \emph{any rearrangement-invariant function space}, provided that the potential inequalities are restricted to vector-valued  functions satisfying \emph{co-canceling} conditions.
\iffalse
This is not only shown to be true for Sobolev inequalities for the standard first or higher order gradient, but also for any other linear homogeneous elliptic \emph{canceling} differential operator. 
\fi
%Since Sobolev inequalities in rearrangement-invariant spaces are fully characterized, 
%As a consequence, a comprehensive theory of optimal Sobolev embeddings, for these operators, in the family of all rearrangement-invariant spaces  is developed.

%\subsection{Riesz potential estimates under co-canceling constraints}
The Riesz potential operator 
$I_\alpha$, with $\alpha \in (0, n)$, is classically defined on locally integrable functions $F: \rn \to \mathbb R^m$ as 
\begin{align*}
I_\alpha F(x) = \frac{1}{\gamma(\alpha)}\int_{\mathbb{R}^n} \frac{F(y)}{|x-y|^{n-\alpha}}\;dy \quad \text{for $x\in \rn$,}
\end{align*}
whenever the integral on the right-hand side is finite. Here, $n, m \in \N$, and 
 $\gamma(\alpha)$ is 
 a suitable normalization constant.

 A central result from \cite{sobolev} asserts that,  if $1<p<n/\alpha$, then there exists a constant $c=c(\alpha,p,n)$ such that
\begin{align}\label{sobolev_ineq}
\|I_\alpha F \|_{L^{\frac {np}{n-\alpha p}}(\mathbb{R}^n, \mathbb R^m)} \leq c \|F\|_{L^p(\mathbb{R}^n,  \mathbb R^m)}
\end{align}
for all $F \in L^p(\mathbb{R}^n, \mathbb R^m)$. The 
 inequality \eqref{sobolev_ineq}
is the key step in  Sobolev's proof of the inequality which, for  $n, \ell, k\in \N$, with $n \geq 2$, and $1<p<\frac nk$, tells us that
\begin{align}
    \label{may1}
    {\|u\|_{L^{\frac{np}{n-k p}}(\rn, \mathbb{R}^\ell)}\leq c \|\nabla ^k u\|_{L^p(\rn, \mathbb{R}^{\ell \times n^k})}}
\end{align}
for some constant $c$ and every function  $u:\mathbb R^n\rightarrow\mathbb R^\ell$ in the homogeneous Sobolev space of 
$k$-times weakly differentiable functions decaying to zero near infinity. Here,  %$$\|u\|_{V^{\alpha}L^p(\rn)}= \|\nabla ^\alpha u\|_{L^p(\rn)}$$
%and 
$\nabla ^k u$ stands for the  tensor of all $k$-th order derivatives of $u$. 
Moreover, throughout this section,   the notion of decay near infinity will   be left unspecified, since it may have slightly different meanings on different occurrences.  
%It will be made precise in the following sections.

 A link between \eqref{sobolev_ineq} and \eqref{may1}  depends on the inequality
\begin{align}
    \label{may2}
    |u(x)|\leq c I_k(|\nabla ^k u|)(x)
%    
%    \int_{\mathbb{R}^n} \frac{|\nabla ^\alpha u(y)|}{|x-y|^{n-\alpha}}\;dy 
\quad \text{for a.e. $x\in \rn$,}
\end{align}
which holds, if $1 \leq k <n$, for some constant $c=c(n, k)$, and every  $u$ as in \eqref{may1}.

As a consequence of the estimate \eqref{may2} and of boundedness properties  of the operator $I_\alpha$ between function spaces of diverse kinds, parallel Sobolev type inequalities involving the same function spaces are available. For instance, an improved version of \eqref{sobolev_ineq}, independently due to O'Neil \cite{oneil} and Peetre \cite{peetre}, reads:
\begin{align}\label{sobolev_lorentz}
\|I_\alpha F \|_{L^{\frac {np}{n-\alpha p},p}(\mathbb{R}^n, \mathbb R^m)} \leq c \|F\|_{L^p(\mathbb{R}^n, \mathbb R^m)}
\end{align}
for some constant $c$ and every $F \in L^p(\mathbb{R}^n, \mathbb R^m)$.  For $k$ and $p$ as in \eqref{may1}, this implies 
the enhanced Sobolev type inequality
\begin{align}
    \label{may3}
   {\|u\|_{L^{\frac{np}{n-k p},p}(\rn, \mathbb{R}^\ell)}\leq c \|\nabla ^k u\|_{L^p(\rn, \mathbb{R}^{\ell\times n^k})}}
\end{align}
for some constant $c$ and for every $u$ as in \eqref{may1}. 
%\todo[inline]{Daniel:  Does this Lorentz space target need to be added?}
Notice that  the Lorentz spaces $L^{\frac {np}{n-\alpha p},p}(\mathbb{R}^n, \rM)$ and $L^{\frac {np}{n-k p},p}(\mathbb{R}^n, \mathbb R^\ell)$  can be shown to be optimal  (i.e. the smallest possible)  among all rearrangement-invariant target spaces in    
\eqref{sobolev_lorentz}  and \eqref{may3}.

So far we have assumed that $p>1$. This restriction is critical, since 
the correspondence between Sobolev and potential estimates 
 of the same order breaks down in the borderline case when $p=1$. Whereas  the inequality \eqref{sobolev_ineq} fails for $p=1$, the inequalities \eqref{may1} and \eqref{may3}
classically continue to hold via different approaches, such as that of \cite{gagliardo} and \cite{nirenberg} based on one-dimensional integration and H\"older's inequality, or that of \cite{mazya} and \cite{federer}, relying upon the isoperimetric theorem and the coarea formula.

%\todo[inline]{Daniel:  The following paragraph now seems a little out of place.  It connects directly with the paragraph after the next, maybe we can move it or edit/revise the paragraph in between?}
A substitute for  \eqref{sobolev_ineq} involves the Marcinkiweicz space $L^{\frac n{n-\alpha}, \infty}(\mathbb{R}^n, \mathbb R^m)$, also called weak Lebesgue space $L^{\frac n{n-\alpha}}(\mathbb{R}^n, \mathbb R^m)$, on the left-hand side. Namely, one has  
\begin{align}\label{weaktype}
\|I_\alpha F \|_{L^{\frac {n}{n-\alpha},\infty}(\mathbb{R}^n, \mathbb R^m)} \leq c \|F\|_{L^1(\mathbb{R}^n, \mathbb R^m)}
\end{align}
for some constant $c$ and every $F \in L^1(\mathbb{R}^n, \mathbb R^m)$. Moreover, the space $L^{\frac {n}{n-\alpha},\infty}(\mathbb{R}^n, \mathbb R^m)$ is the best (smallest) possible among all rearrangement-invariant target spaces in \eqref{weaktype}.

The results recalled above are prototypical examples of a general principle, which, loosely speaking, affirms that Riesz potential and Sobolev inequalities of the same order share the same domain and target spaces when the former is not too close to $L^1$, but, if this is not the case, then a Riesz potential inequality is essentially weaker than its Sobolev analog. As will be clear from our discussion, this qualitative statement can be made precise whenever 
Sobolev spaces, possibly of fractional-order, associated with 
rearrangement-invariant norms are concerned.

  In this connection,
a  new phenomenon was discovered in the series of works \cites{Strauss, BourgainBrezisMironescu, BourgainBrezis2004, BourgainBrezis2007, DG,DG2,GRV,LanzaniStein, RaitaSpector, RSS, Spector-VanSchaftingen-2018,VS,VS2,VS2a,VS3,VS4,SSVS}.  In particular, from  \cite[Proposition 8.7]{VS3} it can be deduced that
 the estimate \eqref{weaktype} upgrades to a 
  strong-type estimate, 
where the Marcinkiewicz norm is replaced by the norm in the   Lebesgue space $L^{\frac n{n-\alpha}}(\mathbb{R}^n,\mathbb{R}^m)$, provided that $m\geq 2$ and the admissible
functions $F \in L^1(\mathbb{R}^n,\mathbb{R}^m)$ are subject to the constraint 
\begin{align}
    \label{constr}
    \mathcal L F=0
\end{align}
in the distributional sense,
where 
$\mathcal L$ is any linear homogeneous co-canceling differential operator of order $k\in \N$. Basic, yet important, instances of first-order operators of this kind are provided by the divergence and the curl. The higher-order divergence operator also belongs to this class. The general notion of co-canceling operator was introduced in \cite[Definition 1.3]{VS3} and is recalled at the beginning of Section \ref{main}.

 A further improvement of the result from \cite{VS3} asserts that, under the constraint \eqref{constr}, the inequality \eqref{sobolev_lorentz} also holds  for $p=1$. 
Namely, if $m \geq 2$, then 
\begin{align}\label{false1'}
\|I_\alpha F \|_{L^{\frac n{n-\alpha}, 1}(\mathbb{R}^n,\mathbb{R}^m)} \leq c \|F\|_{L^1(\mathbb{R}^n,\mathbb{R}^m)}
\end{align}
for some constant $c$ and every function $F \in L^1_{\mathcal L}(\mathbb{R}^n,\mathbb{R}^m)$. Here, the notation $L^1_{\mathcal L}(\mathbb{R}^n,\mathbb{R}^m)$ stands for the space of those functions $F\in L^1(\mathbb{R}^n,\mathbb{R}^m)$ which  fulfill  \eqref{constr}. 
This inequality is established in \cites{HS,HRS} for first order co-canceling operators and  in  \cite{Stolyarov} in the higher order case.

We shall show that such a phenomenon is not peculiar of the domain space $L^1(\mathbb{R}^n,\mathbb{R}^m)$, but surfaces for any rearrangement-invariant domain and target spaces. This is the content of
%a consequence of our result about Riesz potential inequalities for vector fields $F$ subject to the co-canceling condition \eqref{constr}, contained in 
Theorem \ref{sobolev}, which   tells us that, if $X(\rn, \mathbb R^m)$ and  $Y(\rn, \mathbb R^m)$ are rearrangement-invariant spaces and $\alpha \in (0,n)$, then the inequality:
\begin{align}\label{may6}
\|I_\alpha F \|_{Y(\mathbb R^n, \mathbb R^m)} \leq c \|F\|_{X(\mathbb{R}^n, \mathbb R^m)}
\end{align}
for some   constant $c$ and all $F\in X_{\mathcal L}(\mathbb{R}^n, \mathbb R^m)$ is a consequence of the  much simpler one-dimensional Hardy type inequality:
\begin{equation}\label{bound1intro}
\bigg\|\int_s^\infty r^{-1+\frac \alpha n}f(r)\, dr \bigg\|_{Y(0,\infty)} \leq c \|f\|_{X(0, \infty)}
\end{equation}
for every $f \in X(0,\infty)$,  in the  respective one-dimensional representation spaces $X(0, \infty)$ and $Y(0, \infty)$.
\\ On the other hand, the Hardy inequality \eqref{bound1intro}, with $\alpha =k\in \N$ and  $1 \leq k <n$, 
\iffalse
namely 
\begin{equation}\label{bound1k}
\bigg\|\int_s^\infty r^{-1+\frac k n}f(r)\, dr \bigg\|_{Y(0,\infty)} \leq c \|f\|_{X(0, \infty)}
\end{equation}
 for  $f \in X(0,\infty)$,
 \fi
is necessary and sufficient for the Sobolev  inequality
%\todo[inline]{Dominic: Shouldn't we give a reference for this?}
\begin{align}
    \label{may5}
\|u\|_{Y(\rn, \mathbb{R}^{\ell})}\leq c \|\nabla ^ku\|_{X(\rn, \mathbb{R}^{\ell\times n^k})}
\end{align}
%\todo[inline]{Daniel:  One can decide whether to insert the target spaces here.}
to hold 
 for some constant $c$ and every $k$-times weakly differentiable function $u$ decaying to zero near infinity.

 %This is shown in \cite[Theorem 2.3]{mihula} for functions $u$ from the space $V^{k}_{\rm d}X(\rn)$ which decay near infinity in a weaker sense than that to be adopted in this paper. An inspection
  
 Therefore, loosely speaking, a central message of this contribution is that, if   $X(\rn, \rM)$ and  $Y(\rn, \rM)$ are rearrangement-invariant spaces and $\alpha =k \in \N$, then:
  \begin{verse}
  \emph{\lq\lq The constrained Riesz potential  inequality \eqref{may6}  holds  
  whenever the Sobolev inequality \eqref{may5}
holds."}
\end{verse}
This principle is true  even for non-integer $\alpha$,  provided that a fractional space of Gagliardo-Slobodeskji type of order $\alpha$, built upon the norm in $X(\mathbb{R}^n, \rM)$, is well defined. Besides the Lebesgue norms, to which the classical Gagliardo-Slobodeskji spaces are associated, this happens, for instance, when $X(\mathbb{R}^n, \rM)$ is an Orlicz space -- see \cite{ACPS}.
Indeed, 
 the Sobolev inequality \eqref{may5} and the relevant one-dimensional inequality are equivalent also  in this case \cite[Theorem 3.7]{ACPS_NA}.

Riesz potential inequalities  under co-canceling constraints as in  \eqref{may6}, for quite general families of rearrangement-invariant spaces $X(\mathbb{R}^n, \mathbb R^m)$ and $Y(\mathbb{R}^n, \mathbb R^m)$, are presented in Section \ref{main}. 
%The corresponding Sobolev inequalities of the form \eqref{sobolev1intro} for elliptic canceling operators are displayed in Section \ref{S:sobolev}.
Here, we content ourselves with exhibiting their implementation to specific special instances.

It is  clear from the discussion above that our results are most relevant for borderline spaces $X(\mathbb{R}^n, \mathbb R^m)$
 which are \lq\lq close" to $L^1(\mathbb{R}^n, \mathbb R^m)$.
  Otherwise, the  inequality \eqref{may6}, even without the constraint \eqref{constr},  holds with the same spaces as in the Sobolev inequality \eqref{may5}.  

%{\color{magenta} In turn, the   inequality \eqref{may6} without the co-canceling constraint implies 
%the inequality \eqref{sobolev1intro}
%with the same norms as in %\eqref{may5}. This can be deduced 
% from a representation formula for functions $u$ in terms of $\mathcal A_k(D)u$ via an integral operator associated with a kernel having a Riesz potential type growth.}

%In order to give an  idea   of the inequalities that will be derived,
Thus, below we single out  a few illustrative examples concerning  spaces neighboring $L^1$. To begin with,
 consider
  a perturbation of the $L^1$ norm by a logarithmic factor. Given any $r\geq 0$, we have that
\begin{align}
    \label{exlog}
    \|I_\alpha F\|_{L^{\frac n{n-\alpha}}(\log L)^{\frac {nr}{n-\alpha}}(\rn, \rM)} \leq c \|F\|_{L^1(\log L)^{r}(\rn, \rM)}
\end{align}
for some constant $c$ and every $F\in L^1(\log L)^{r}_{\mathcal L}(\rn, \rM)$. This is a special case of Example \ref{ex1}, Section \ref{main}. Here, $L^1(\log L)^r(\rn, \rM)$ denotes the Orlicz space associated with a Young function of the form $t(\log (b+t))^r$, and  $b$ is sufficiently large to ensure that this function is convex.

Let us emphasize that the inequality \eqref{exlog}, as well all the Riesz potential inequalities in the remaining part of this section,  fail  if the constraint \eqref{constr} is dropped.

 Next, denote by $L^1(\log \log L)^{r}(\rn, \rM)$   the Orlicz space built upon a Young function of the form $t(\log \log (b+t))^r$ for sufficiently large $b$. From Example \ref{ex2}, Section \ref{main}, we obtain that
\begin{align}
    \label{exloglog}
    \|I_\alpha F\|_{L^{\frac n{n-\alpha}}(\log \log L)^{\frac {nr}{n-\alpha}}(\rn, \rM)} \leq c \|F\|_{L^1(\log \log L)^{r}(\rn, \rM)}
\end{align}
for some constant $c$ and  every $F\in L^1(\log \log L)^{r}_{\mathcal L}(\rn, \rM)$.

 The estimates \eqref{exlog} and \eqref{exloglog}
admit  improvements in the framework of Lorentz-Zygmund target norms. 
In particular, the following inequality holds:
\begin{align}
    \label{exLZ}
    \|I_\alpha F\|_{L^{\frac n{n-\alpha},1,r}(\rn, \rM)} \leq c \|F\|_{L^1(\log L)^{r}(\rn, \rM)}
\end{align}
for some constant $c$ and every $F\in L^1(\log L)^{r}_{\mathcal L}(\rn, \rM)$, see Example \ref{ex orlicz-lorentz}, Section \ref{main}.   Observe that this inequality actually  improves  \eqref{exlog}, inasmuch as the Lorentz-Zygmund space
 $L^{\frac n{n-\alpha},1,r}(\rn, \rM) \subsetneq L^{\frac n{n-\alpha}}(\log L)^{\frac {nr}{n-\alpha}}(\rn, \rM)$.  A parallel improvement of the inequality \eqref{exloglog} can be obtained via Theorem \ref{sobolev-orliczlorentz}, Section \ref{main}. 

 Further examples of  constrained potential inequalities in  borderline spaces involve  Lorentz-Zygmund domain spaces, with first exponent equal to $1$. This requires   restricting to functions $F$ vanishing outside a set $\Omega \subset \rn$ with finite measure and to consider norms over $\Omega$,
 since the Lorentz-Zygmund spaces in question are trivially equal to $\{0\}$ on sets with infinite measure. Example \ref{lorentzzygmund}, Section \ref{main}, asserts that,
 given $q\in [1,\infty)$ and $r>-\frac 1q$, 
 \begin{align}
    \label{exLZdom}
    \|I_\alpha F\|_{L^{\frac n{n-\alpha},q,r+1}(\Omega, \rM)} \leq c \|F\|_{L^{(1,q,r)}(\Omega, \rM)}
\end{align}
for some constant $c$ and every $F\in L^{(1,q,r)}_\mathcal L(\rn, \rM)$ vanishing outside $\Omega$.   Here, 
$L^{(1,q,r)} L(\Omega, \rM)$ denotes a Lorentz space endowed with a norm defined via the maximal function associated with the decreasing rearrangement, instead of the plain decreasing rearrangement.
Analogous results  can be derived also for   $r\leq -\frac 1q$. However, in this range of values of the parameters,   the optimal target spaces belong to even more general classes. For simplicity, we prefer not to enter this question.

 Let us note that, when $\alpha \in \N$, the target spaces in \eqref{exlog} and \eqref{exloglog} are known to be   optimal among all Orlicz spaces 
 in the Sobolev inequalities of order $\alpha$ with the same domain spaces. An analogous property holds with regard to the target space in \eqref{exLZ} and (at least for $q>1$) also in \eqref{exLZdom}  in the class of rearrangement-invariant spaces. We refer to \cite{cianchi_ibero, cianchi_forum} for the first three Sobolev inequalities and \cite{cavaliere-new} for the last one. 

 The key tool in our approach to  constrained Riesz potential inequalities is an estimate, in rearrangement form, for  Riesz potentials of arbitrary order divergence free vector fields. This is the content of Theorem \ref{K-CZ}, which is a special case of \cite[Theorem 5.1]{BCS_canceling}, which also holds for the composition of $I_\alpha$ with linear operators from certain classes. A combination of Theorem \ref{K-CZ} with 
 specific properties of rearrangement-invariant spaces  yields the main result of this paper, contained in Theorem \ref{sobolev}, in the special case when the co-canceling operator is the divergence operator of any order. The case of a general  co-canceling operator $\mathcal L$ is reduced to the latter in Lemma \ref{reductionk}, 
 via a result of \cite{HRS}. Although, as mentioned above,  Theorem \ref{K-CZ} follows through results from \cite{BCS_canceling}, we take the opportunity here to offer a self contained proof for the classical first-order divergence operator. 
  It involves somewhat simpler notation and arguments, and just focuses on the 
Riesz potential operator.

\section{Function-space background}\label{back}

%\todo[inline]{A: we have to check if something can be deleted in this section}
%  This section is devoted to basic definitions and properties concerning functions and function spaces playing a role in the paper.
  
Throughout the paper, the relation $\lq\lq \lesssim "$ between two positive expressions means that the former is bounded by the latter, up to a positive multiplicative constant depending on quantities to be specified.
The relations  $\lq\lq \gtrsim "$ and $\lq\lq \approx "$ are defined accordingly. 
%have  between two expressions means that   they are bounded by each
%other  up to multiplicative constants depending on quantities to be specified.

The notation $|E|$ is adopted for the Lebesgue measure of a set $E\subset \rn$. Let $\Omega $ be a measurable subset of $\rn$, with $n \in \mathbb N$, and let $m \in \mathbb N$.  We denote by 
$\M(\Omega, \mathbb R^m)$ 
 the space of all Lebesgue-measurable functions
$F : \Omega \to \mathbb R^m$. When $m=1$, we shall simply write $\M(\Omega)$. A parallel convention will be adopted for other function spaces. 
%Usually, vector-valued functions will be denoted by capital letters, whereas lower case letters will be adopted in the scalar-valued case.
 Moreover, we  set $\Mpl(\Omega)=\{f\in\M(\Omega)\colon f\geq 0 \
\textup{a.e. in}\ \Omega\}$.

The \emph{decreasing rearrangement}  $F^{\ast}:[0, \infty )\to
[0,\infty]$ of a function $F \in \M(\Omega, \mathbb R^m)$ is   given by
\begin{equation*}
F^{\ast}(s) = \inf \{t\geq 0: |\{x\in \Omega : |F(x)|>t \}|\leq s \}
\qquad \hbox{for $s \in [0,\infty)$}.
\end{equation*}
The function  $F^{**}: (0, \infty ) \to [0, \infty )$ is defined as
\begin{equation}\label{c15}
F^{**}(s)=\frac{1}{s}\int_0^s F^*(r)\,dr \qquad \hbox{for $s>0$}.
\end{equation}
One has
\begin{align}
   \label{aug11}
    F^{**}(s)= \frac 1s \sup\bigg\{ \int_E|F|\,dx: E\subset \Omega, |E|=s\bigg\}.
\end{align}
%The Hardy-Littlewood inequality tells us that
%\begin{equation}\label{HL}
%\int_{\Omega}|F(x)| |G(x)|dx\leq %\int_0^\infty F^*(s) G^*(s)\, ds
%\end{equation}
%for $F, G\in \mathcal (\Omega, \rM)$.
Assume that $L\in (0,\infty]$. A 
\textit{function norm} on $(0, L)$ is a functional
 $\|\cdot\|_{X(0,L)}{:\Mpl(0,L)\to[0,\infty]}$ such that, for all functions $f, g\in \Mpl(0,L)$, all sequences
$\{f_k\} \subset {\Mpl(0,L)}$, and every $\lambda \geq 0$:
\begin{itemize}
\item[(P1)]\quad $\|f\|_{X(0,L)}=0$ if and only if $f=0$ a.e.;
$\|\lambda f\|_{X(0,L)}= \lambda \|f\|_{X(0,L)}$; \par\noindent \quad
$\|f+g\|_{X(0,L)}\leq \|f\|_{X(0,L)}+ \|g\|_{X(0,L)}$;
\item[(P2)]\quad $ f \le g$ a.e.\  implies $\|f\|_{X(0,L)}
\le \|g\|_{X(0,L)}$;
\item[(P3)]\quad $f_k \nearrow f$ a.e.\
implies $\|f_k\|_{X(0,L)} \nearrow \|f\|_{X(0,L)}$;
\item[(P4)]\quad $\|\chi _E\|_{X(0,L)}<\infty$ if $|E| < \infty$;
\item[(P5)]\quad if $|E|< \infty$, then there exists a constant
 $c$, depending on $E$ and $X(0,L)$, such that \\   $\int_E f(s)\,ds \le c
\|f\|_{X(0,L)}$.
\end{itemize}
Here, $E$ stands for a measurable set in $(0,L)$, and  $\chi_E$  its characteristic function.
Under the additional assumption that
\begin{itemize}
\item[(P6)]\quad $\|f\|_{X(0,L)} = \|g\|_{X(0,L)}$ whenever $f\sp* = g\sp *$,
\end{itemize}
the functional $\|\cdot\|_{X(0,L)}$ is called a
\textit{rearrangement-invariant function norm}.
\\
The \textit{associate function norm}  $\|\cdot\|_{X'(0,L)}$ of a function norm $\|\cdot\|_{X(0,L)}$ is  defined as
$$
\|f\|_{X'(0,L)}=\sup_{\begin{tiny}
                        \begin{array}{c}
                       {g\in{\Mpl(0,L)}}\\
                        \|g\|_{X(0,L)}\leq 1
                        \end{array}
                      \end{tiny}}
\int_0^{L}f(s)g(s)ds
$$
for $ f\in\Mpl(0,L)$.

Assume that $\Omega$ is a measurable set in $\rn$, and let
 $\|\cdot\|_{X(0,|\Omega|)}$ be   a rearrangement-invariant function norm.  Then the space $X(\Omega, \rM)$ is
defined as the set of all  functions  $F \in\M(\Omega, \rM)$
for which the expression
\begin{equation}\label{norm}
\|F\|_{X(\Omega, \rM)}= \|F^*\|_{X(0,|\Omega|)} 
%
%begin{cases}  \|u\sp*_\nu\|_{X(0,\infty)} & \quad \hbox{if $\nu (\mathcal R)= \infty$}
%\\
%\|\phi\sp*_\nu(\nu(\mathcal R)t)\|_{X(0,1)} & \quad \hbox{if $\nu (\mathcal R)< \infty$,}
%\end{cases}
\end{equation}
is finite. The space $X(\Omega, \rM)$ is a Banach space, equipped
with the norm defined as \eqref{norm}. 
 The space $X(0,|\Omega|)$ is called the
\textit{representation space} of $X(\Omega, \rM)$.
\\
The \textit{associate
space}  $X'(\Omega, \rM)$
of   $X(\Omega, \rM)$ is
 the
rearrangement-invariant space   associated with the
function norm $\|\cdot\|_{X'(0,|\Omega|)}$.
%By property \eqref{X''}, $X''(\Omega)=X(\Omega)$.  
%$Hence, any rearrangement-invariant
%space $X(\Omega)$ is always the associate space of another
%rearrangement-invariant space, namely $X'(\Omega)$.
\\
The \textit{H\"older type inequality}
%\[
%\int_{0}\sp1f(t)g(t)\,dt\leq\|f\|_{X(0,L)}\|g\|_{X'(0,L)},
%\]
%holds for every $f,g\in\Mpl(0,L)$, and hence
\begin{equation}\label{holder}
\int_{\Omega}|F| |G|dx\leq\|F\|_{X(\Omega, \rM)}\|G\|_{X'(\Omega, \rM)}
\end{equation}
holds  for every $F\in X(\Omega, \rM)$ and $G \in X'(\Omega, \rM)$.
%\\ The norm $\|\cdot\|_{X(\Omega, \rM)}$ is said to be absolutely continuous if for every $F\in X(\Omega, \rM)$ and every sequence of sets $\{E_k\}$, with $E_k\subset \Omega$ and $E_k \to \emptyset$, one has that $$\|F\chi_{E_k}\|_{X(\Omega, \rM)}\to 0.$$
 \\ Hardy's lemma tells us that
\begin{align}
    \label{hardy}
    \text{if\,\, $F^{**}(s)\leq G^{**}(s)$,\,\, then \,\,$\|F\|_{X(\Omega, \rM)}\leq \|G\|_{X(\Omega, \rM)}$}
\end{align}
for every rearrangement-invariant space $X(\Omega, \rM)$ and for every $F, G \in \mathcal M(\Omega, \rM)$. 
\\
Let $X(\Omega, \rM)$ and $Y(\Omega, \rM)$ be rearrangement-invariant
spaces. The notation $X(\Omega, \rM) \to Y(\Omega, \rM)$ means that
$X(\Omega, \rM)$ is continuously embedded into $Y(\Omega, \rM)$; namely,
$\|F\|_{Y(\Omega, \rM)}\leq c\|F\|_{ X(\Omega, \rM)}$ for some constant   $c$ and every $F\in X(\Omega, \rM)$.
\\
Let $X(\rn, \rM)$ be a rearrangement-invariant space. Then
\begin{equation}\label{dec2}
{L^1(\rn, \rM)\cap L^\infty(\rn, \rM)} \to
X(\rn, \rM) \to {L^1(\rn, \rM)+L^\infty(\rn, \rM)}.
\end{equation}
If  $|\Omega|< \infty$, then
\begin{equation}\label{l1linf}
L^\infty (\Omega, \rM) \to X(\Omega, \rM) \to L^1(\Omega, \rM)
\end{equation}
for every rearrangement-invariant space
$X(\Omega, \rM)$.
\\ 
The following definition canonically produces 
a rearrangement-invariant space
 on the whole of $\rn$ from a 
rearrangement-invariant space $X(\Omega, \rM)$  on a measurable set $\Omega \subset \rn$. Let
 $\|\cdot\|_{X^e(0, \infty)}$ be the function norm given by 
\begin{equation}\label{sep25}
\|f\|_{X^e(0, \infty)} = \|f^*\|_{X(0, |\Omega|)}
\end{equation}
for $f \in \mathcal M_+(0,\infty)$.
Then define $X^e(\R^n, \rM)$ as the rearrangement-invariant space built upon the function norm $\|\cdot\|_{X^e(0, \infty)}$. Plainly,
\begin{equation}\label{feb83}
\|F\|_{X^e(\R^n, \rM)}=  \|F^*\|_{X(0, |\Omega|)}
\end{equation}
for every $F \in \mathcal M(\R^n, \rM)$. Furthermore, if $F =0$ a.e. in $\rn \setminus \Omega$, then 
\begin{equation}\label{extbis}
\|F\|_{X^e(\R^n, \rM)}=  \|F\|_{X(\Omega, \rM)}.
\end{equation}

Let $\{\rho_h\}$ be a family of smooth mollifiers, namely $\rho_h \in C^\infty_c(B_{1/h})$,   $\rho_h \geq 0$, and $\int_{\mathbb R^n}\rho_h (x)\,dx =1$ for $h \in \N$. Here, $B_R$ denotes the ball, centered at $0$, with radius $R$.
The convolution $F_h = F *\rho_h$ of a function $F \in L^1_{\rm loc}(\mathbb R^n, \rM)$ with $\rho_h$ satisfies the inequality:
\begin{equation}\label{nov1}
\int_0^s F_h^* (\tau) \, d\tau \leq \int_0^s F^* (\tau) \, d\tau \qquad \text{for $s \geq 0$,}
\end{equation}
 for every $h \in \N$.  Consequently,
\begin{equation}\label{nov2}
\|F_h\|_{X(\mathbb R^n, \rM)} \leq \|F\|_{X(\mathbb R^n, \rM)}
\end{equation}
for every  rearrangement-invariant space $X(\mathbb R^n, \rM)$ and every  $h\in \N$.

\iffalse
\begin{theoremalph}\label{convolution} Let $\rho_h \in C^\infty_c(B_{\frac 1k} (0))$ be such that $\rho_h \geq 0$ and $\int_{\mathbb R^n}\rho_h (x)\,dx =1$ for $h \in \N$. Given $F \in L^1_{\rm loc}(\mathbb R^n, \rM)$, set
$$F_k = F *\rho_h.$$
Then
\begin{equation}\label{nov1}
\int_0^t F_k^* (s) \, ds \leq \int_0^t F^* (s) \, ds \qquad \text{for $t \geq 0$,}
\end{equation}
for every $k \in N$.
Hence,
\begin{equation}\label{nov2}
\|F_\epsilon\|_{X(\mathbb R^n)} \leq \|F\|_{X(\mathbb R^n)}
\end{equation}
for every  rearrangement-invariant space $X(\mathbb R^n)$ and every  $k\in \nabla$.
\end{theoremalph}
\fi
\iffalse
\begin{proof} By O'Neil's convolution inequality,
\begin{equation}\label{nov3}
F_\epsilon ^{**}(t)\leq t F^{**}(t) \rho_\epsilon^{**}(t) + \int_t^\infty F^*(s)  \rho_\epsilon^{*}(s)\, ds \quad \text{for $t >0$.}
\end{equation}
Thus
    \begin{align}\label{approx3}
      F_\epsilon^{**} (t) & \leq F^{**}(t)\int _0^t \rho_\epsilon^{*}(s)\, ds
      + \int _t ^\infty F^*(s) \rho_\epsilon^{*}(s)\, ds
    \\ \nonumber & \leq F^{**}(t) \int _0^t  \rho_\epsilon^{*}(s)\, ds + F^*(t ) \int _t ^\infty   \rho_\epsilon^{*}(s)\, ds
   \\ \nonumber & \leq F^{**}(t) \int _0^\infty  \rho_\epsilon^{*}(s)\, ds
    \\ \nonumber & = F^{**}(t) \int _{\rn} \rho_\epsilon (x) dx =
   F^{**}(t) \quad \hbox{for $t >0,$}
      \end{align}
     namely inequality \eqref{nov1}.
\end{proof}
\fi

The $K$-functional of a couple of normed spaces $(Z_0, Z_1)$, which are both continuously embedded into some Hausdorff vector space, is defined as 
\begin{align}\label{Kfunct}
K(\zeta, t; Z_0, Z_1) = \inf \big\{  \|\zeta_0\|_{Z_0} + t \|\zeta_1\|_{Z_1}: \zeta = \zeta_0 + \zeta_1, \, \zeta_0 \in Z_0, \, \zeta_1 \in Z_1\big\} \quad \text{for $t>0$,}
%_{\begin{tiny} \begin{array}{c}
                %       {\zeta = \zeta_0 + \zeta_1}\\
                 %       {\zeta_0 \in Z_0, \, \zeta_1 \in Z_1}
                  %    \end{array}
                  %    \end{tiny}} \big(\|\zeta_0\|_{Z_0} + t \|\zeta_1\|_{Z_1}\big) \qquad \text{for $t>0$.}
\end{align}
for every $\zeta \in Z_0 + Z_1$.
\iffalse
 If $0<\theta <1$ and $1\leq q <\infty$, the interpolation space $(Z_0,Z_1)_{\theta,q}$ consists of all
 $\zeta \in Z_0 +Z_1$ for which the functional
 \begin{equation}\label{dec38}
 \|\zeta\|_{(Z_0,Z_1)_{\theta,q}} = \bigg(\int_0^\infty \big(t^{-\theta} K(\zeta, t; Z_0, Z_1)\big)^q\frac {dt}t\bigg)^{\frac 1q}
 \end{equation}
 is finite.\fi
If $n, m\in \mathbb N$ and   $\|\cdot \|_{X(0,\infty)}$ and  $\|\cdot \|_{Y(0,\infty)}$ are rearrangement-invariant function norms,
then
\begin{equation}\label{K-v1}
K(|F|, t; X(\rn), Y(\rn)) =K(F, t; X(\rn, \mathbb R^m), Y(\rn, \mathbb R^m)) 
%K(|\bfU|, t; X(\Omega), Y(\Omega))\leq  K(\bfU, t; X(\Omega, \mathbb R^m), Y(\Omega, \mathbb R^m)) 
%\leq m K(|\bfU|, t; X(\Omega), Y(\Omega))
\quad \text{for $t>0$,}
\end{equation}
for $F\in X(\rn, \mathbb R^m) +Y(\rn, \mathbb R^m)$. See \cite[Lemma 7.3]{Breit-Cianchi} for a proof.

Let $p\in [1,\infty)$ and $q\in [1,\infty]$. We denote by  $L^{(p,q)}(\Omega, \rM)$  the Lorentz space associated with the rearrangement invariant  function norm given by
\begin{align}
    \label{lorentz}
    \|f\|_{L^{(p,q)}(0,|\Omega|)}= 
    \big\|s^{1/p-1/q}f^{**}(s)\|_{L^q(0,|\Omega|)}
%    
   % \bigg(\int_0^L\big(f^{**}(s)s^{1/p}\big)^q\frac {ds}s\bigg)^{\frac 1q}
\end{align}
for $f\in \Mpl (0,|\Omega|)$. 
%If $p \in (1, \infty)$, then the norm \eqref{lorentz} is equivalent, up to multiplicative constants depending on $p$ and $q$, to the functional defined as 
Replacing $f^{**}$ with $f^*$ in \eqref{lorentz} results 
in the functional $\|\cdot\|_{L^{p,q}(\Omega, \rM)}$ and the space $L^{p,q}(\Omega, \rM)$. 
%
%By $L^{p,q}(\Omega, \rM)$ we denote the space of those functions such that
%\begin{align}
%    \label{nov105}
    %\|f\|_{L^{p,q}(0,|\Omega|)}=
    %\big\|s^{1/p-1/q}f^{*}(s)\|_{L^q(0,|\Omega|)}<\infty,
%\end{align}
%and by $\|F\|_{L^{p,q}(\Omega, \rM)(0,|\Omega|)}$ the functional defined as in \eqref{nov105}, with $f=|F|$.
    %for $f\in \Mpl (0,|\Omega|)$. 
    This functional is a norm if $1\leq q\leq p$, and  is equivalent to the norm $\|\cdot\|_{L^{(p,q)}(\Omega, \rM)}$ if $p\in (1,\infty)$.
    \\
    The class of Lorentz spaces includes that of Lebesgue spaces, since $L^{(p,p)}(\Omega, \rM)= L^p(\Omega, \rM)$, up to equivalent norms, for $p\in (1,\infty)$, and $L^{p,p}(\Omega, \rM)= L^p(\Omega, \rM)$ for $p\in [1,\infty)$.
    \\ If $q<\infty$, then the norm $\|\cdot\|_{L^{(p,q)}(\Omega, \rM)}$ is absolutely continuous.
    \\ 
    %Since $(L^{p,q})'(\Omega, \rM)= L^{p',q'}(\Omega, \rM)$, up to equivalent norms, for $p\in (1,\infty)$,the inequality \eqref{holder} yields:\todo[inline]{Daniel:  I believe this also requires $q<+\infty$ since the dual of $L^{p,\infty}$ is not $L^{p',1}$ }
 For $p\in (1,\infty)$, the H\"older inequality in Lorentz spaces takes the form
\begin{align}
    \label{holderlor}
    \int_{\Omega} |F||G|  \, dx \leq  \|F\|_{L^{p,q}(\Omega, \rM)} \|G\|_{L^{p',q'}(\Omega, \rM)}
\end{align}
 for  $F\in L^{p,q}(\Omega, \rM)$ and $G\in L^{p',q'}(\Omega, \rM)$. Here, $p'$ and $q'$ denote the H\"older conjugates of $p$ and $q$.
\\ 
Moreover, if $p\in (1,\infty)$, then 
\begin{align}
    \label{aug3}
\|F\|_{L^{p,q}(\rn, \rM)} \leq c \|F\|_{L^1(\rn, \rM)}^{\frac 1p} \|F\|_{L^\infty(\rn, \rM)}^{\frac 1{p'}}
\end{align}
for some constant $c$ and 
for every $F\in {L^1(\rn, \rM)\cap L^\infty(\rn, \rM)}$.
\iffalse
\todo[inline]{A: is there a reference for \eqref{aug3}?}
Indeed, the first embedding in \eqref{dec2} ensures that there exists a constant $c$ such that
\begin{align}
    \label{aug4}
    \|F\|_{L^{p,q}(\rn, \rM)} \leq c \big(\|F\|_{L^1(\rn, \rM)}+\|F\|_{L^\infty(\rn, \rM)} \big)
\end{align}
for  $F\in {L^1(\rn, \rM)\cap L^\infty(\rn, \rM)}$. Given $\lambda >0$, an application of the inequality \eqref{aug4} with $F$ replaced with the function $F_\lambda$, given by $F_\lambda(x)=F(\lambda x)$ for  $x\in \rn$, enables us to deduce that
\begin{align}
    \label{aug5}
    \|F\|_{L^{p,q}(\rn, \rM)} \leq c \big(\lambda ^{-\frac n{p'}}\|F\|_{L^1(\rn, \rM)}+\lambda ^{\frac np}\|F\|_{L^\infty(\rn, \rM)} \big)
\end{align}
for  $F\in {L^1(\rn, \rM)\cap L^\infty(\rn, \rM)}$. Here, we have made use of the fact that $F_\lambda ^*(s)=F^*(\lambda ^ns)$ for $s \geq 0$.
Choosing $\lambda$ in such a way that the addends in brackets on the right-hand side of the inequality  \eqref{aug5} agree yields \eqref{aug3}.
\fi
\\
 The Lorentz-Zygmund spaces $L^{(p,q, \alpha)}(\Omega, \rM)$ further extend the Lorentz spaces. They are associated with the function norm defined, for $p\in [1, \infty)$, $q\in [1, \infty]$, and $r \in \mathbb R$,   as
\begin{equation}\label{LZ}
\|f\|_{L\sp{(p,q,r)}(0,|\Omega|)}=
\left\|s\sp{\frac{1}{p}-\frac{1}{q}}(1+\log_+ (\tfrac{1}{s}))^r f^{**}(s)\right\|_{L\sp q(0,|\Omega|)}
\end{equation}
for  $f \in {\Mpl(0,|\Omega|)}$. Here, $\log_+$ stands for the positive part of $\log$.  
Replacing $f^{**}$ with $f^*$ in \eqref{LZ} yields the functional  $\|\cdot\|_{L^{p,q,r}(\Omega, \rM)}$ and the corresponding space $L^{p,q,r}(\Omega, \rM)$.  If $p\in (1,\infty)$, the new functional is equivalent (up to multiplicative constants) to $\|\cdot\|_{L^{(p,q,r)}(\Omega, \rM)}$.
\\ The
generalized Lorentz-Zygmund spaces will come into play
in a few borderline inequalities. They are built upon the functional given by
\begin{equation}\label{GLZ}
\|f\|_{L^{p,q,r,\varrho}(0,|\Omega|)}=
\left\|s\sp{\frac{1}{p}-\frac 1q}(1+\log_+ (\tfrac{1}{s}))^{r} (1+\log_+ (1 + \log_+ \tfrac{1}{s})))^{\varrho}f^{*}(s)\right\|_{L\sp q(0,|\Omega|)}
\end{equation}
for  $f \in {\Mpl(0,|\Omega|)}$, where $p,q,r$ are as above and $\varrho \in \mathbb R$.
%
%Like for the plain Lorentz spaces, if $p\in (1, \infty)$, replacing $f^{**}$ with $f^*$
%in \eqref{LZ} results in the functional
%\begin{equation}\label{LZ*}
%\left\|s\sp{\frac{1}{p}-\frac{1}{q}}(1+\log_+ (\tfrac{1}{s}))^r f^{*}(s)\right\|_{L\sp q(0,|\Omega|)}.
%\end{equation}
%The space $L\sp{p,q,r}(\Omega, \rM)$. The 
\\
The  Orlicz spaces extend the family of Lebesgue spaces in a different direction. They are defined
via Young
functions. A \textit{Young function} is a left-continuous convex function  from  $[0, \infty )$ into  $[0, \infty ]$ that vanishes
at $0$ and is not constant in $(0, \infty)$. A Young function $A$ takes the form
\begin{equation}\label{A}
A(t) = \int _0^t a(s)\, ds \quad \quad \textup{for}\ t\in[0,\infty),
\end{equation}
 where $a: [0, \infty )\to
[0, \infty ]$  is a non-decreasing, left-continuous function, which is neither identically equal to $0$, nor to infinity. 
\\
Two Young functions $A$ and $B$ are called equivalent   globally/near infinity/near zero if there exists a positive constant $c$ such that
\begin{align}
    \label{2024-210}
    A(t/c) \leq B(t) \leq A(ct)
\end{align}
for $t\geq 0$/for $t\geq t_0$ for some $t_0>0$/for $0\leq t\leq t_0$ for some $t_0>0$, respectively. The  equivalence between $A$ and $B$ in the sense of \eqref{2024-210} will be denoted as
\begin{align}
\label{2024-211} 
A \simeq B.
\end{align}
For Young functions $A$ and $B$, the relation \eqref{2024-211} implies  that $A \approx B$, but the converse is not true.
\\ The \textit{Orlicz space}
$L\sp A(\Omega, \rM)$ is  defined through the rearrangement-invariant 
\textit{Luxemburg function norm} given
by
\begin{equation}\label{lux}
\|f\|_{L^A(0,|\Omega|)}= \inf \left\{ \lambda >0 :  \int_{0}^{|\Omega|}A \left(
\frac{f(t)}{\lambda} \right) dt \leq 1 \right\}
\end{equation}
for  $f \in {\Mpl(0,|\Omega)}$. For some explicit choices of the function $A$ we shall also employ 
the alternative notation $A(L) 
(\Omega, \rM)$.
%The equi-measurability of $f$ and $f^*$ ensures that  $\|\cdot \|_{L^A (0,L)}$ is actually a rearrangement-invariant function norm.
%\note[inline]{Lubo\v s: Is it clear that with our definition of a Young function the corresponding space is an r.i. space?}
%Lebesgue function norms are easily recovered via suitable choices of the function $A$.
%In particular,
%$L^A (0,1)= L^p (0,1)$ if $A(t)= t^p$ for some $p \in [1, \infty )$,
%and $L^A (0,1)= L^\infty (0,1)$ if $A(t)=0$ for $t\in [0, 1]$ and
%$A(t) = \infty$ for $t>1$.
\\
The  norms $\|\cdot \|_{L^A
(\Omega, \rM)}$ and $\|\cdot \|_{L^B (\Omega, \rM)}$ are equivalent if and only if either $|\Omega|<\infty$ and
$A$ and $B$ are equivalent near infinity, or $|\Omega|=\infty$ and $A$ and $B$ are equivalent globally.
\\ The Lebesgue spaces $L^p(\Omega, \rM)$ are recovered for $A(t)=t^p$ if $p \in [1, \infty)$ and $A(t)=\chi_{(1, \infty)}\infty$ if $p=\infty$.  The Zygmund spaces and the exponential type spaces are further classical examples of Orlicz spaces. The Zygmund spaces $L^p(\log L)^r(\Omega, \rM)$ are built on Young functions of the form $A(t)= t^p(\log (c+t))^r$, where either $p>1$ and $r \in \mathbb R$, or $p=1$ and $r\geq 0$, and $c$ is sufficiently large for $A$ to be convex. If $|\Omega|<\infty$,  one has that  \begin{align}
    \label{eqZYg}
    \|f\|_{L^p(\log L)^r(0, |\Omega|)} \approx \big\|(1+\log_+ (\tfrac{1}{s}))^\frac{r}{p}f^*(s)\big\|_{L^p(0, |\Omega|)}
\end{align}
for $f \in \Mpl (0, |\Omega|)$, up to multiplicative constants independent of $f$ -- see \cite[Lemma 6.12, Chapter 4]{BennettSharpley}. Hence, if $p\in (1,\infty)$, then $L^p(\log L)^r(\Omega, \rM)= L^{p,p, \frac rp}(\Omega, \rM)$, up to equivalent norms.
The exponential spaces $\exp L^r (\Omega, \rM)$, for $r>0$, are associated with Young functions $A(t)$ which are equivalent to $e^{t^r}-1$ near infinity.

%\\ The \textit{Young conjugate} of $A$ is  denoted by  $\widetilde A$.
%and defined as
%$$
%\tilde{A}(t)=\sup\{st - A(s): ~s \geq 0\} \quad \text{for $t\in [0,\infty)$}.
%$$
%The associate function norm  of   $\|\cdot \|_{L^A
%(0,1)}$ is equivalent to $\|\cdot \|_{L^{\widetilde A}
%(0,1)}$, up to absolute multiplicative constants.
%\todo[inline]{Lenka: This is needed in the proof of Theorem~\ref{EX:trace_orlicz-lorentz}.}

\par  The
class of Orlicz-Lorentz spaces embraces diverse instances of Orlicz,  Lorentz, and Lorentz-Zygmund spaces. A specific 
family of Orlicz-Lorentz spaces, which has a role in our applications, is defined as follows.
Given  a Young function $A$ and a number $q\in \R$, we denote by
 $L(A,q)(\Omega, \rM)$ the \emph{Orlicz-Lorentz space}  defined in terms of the functional
\begin{equation}\label{sep35}
	\|f\|_{L(A, q)(0,|\Omega|)}
		= \big\|r^{-\frac{1}{q}}f^{*}(r)\big\|_{L^A(0,|\Omega|)}
\end{equation}
for $f\in \Mpl (0,|\Omega|)$.
 Suitable assumptions on $A$ and $q$ ensure that this functional is  a function norm. This is guaranteed, for example, if  $q>1$ and
\begin{equation}\label{sep36}
\int^\infty \frac{A(t)}{t^{1+q}}\, dt < \infty\,,
\end{equation}
see \cite[Proposition 2.1]{cianchi_ibero}.

\section{Main results}\label{main}

Our  criterion for  Riesz potential inequalities  \eqref{may6} under co-canceling  differential conditions   is stated in 
 Theorem \ref{sobolev} below.
The notion of co-canceling differential operator from \cite[Definition 1.3]{VS3} reads as follows.

\begin{definitionalph}[{\bf Co-canceling operator}]\label{co-canc}
{\rm Let $n, m \geq 2$ and $l \geq 1$. A linear homogeneous $k$-th order constant coefficient  differential operator $\mathcal L(D)$ mapping $\mathbb{R}^m$-valued functions to $\mathbb{R}^l$-valued functions
  is said to be co-canceling if there exist linear operators  $L_\beta: \rM \to \mathbb R^l$, with $\beta \in \mathbb N^n$, such that
\begin{align}\label{operator}
    \mathcal L(D)F = \sum_{\beta \in \mathbb N^n
,\, |\beta|=k}L_\beta(\partial ^\beta F)
\end{align}
for $F \in C^\infty_c(\mathbb{R}^n,\mathbb{R}^m)$, and 
\begin{align*}
\bigcap_{\xi\in\mathbb{R}^n \setminus \{0\}}\ker \mathcal L(\xi)=\{0\},
\end{align*}
where   $\mathcal L(\xi)$ denotes the symbol map of $\mathcal L(D)$ in terms of Fourier transforms.}
\end{definitionalph}

Besides the standard first order divergence and the curl operators, the higher order divergence operator 
 ${\rm div}_k$ is another classical instance of a co-canceling operator. As hinted in Section \ref{intro}, it has a critical role in our approach and is defined as follows.

Let $n\geq 2$ and $k \in\N$. %the $k$-th order divergence %operator acting on functions on $\rn$ which take values in $\R^{n^k}$ will be denoted by 
 %${\operatorname*{div}_k}$.  
 For $F \in C^\infty(\mathbb{R}^n,\mathbb{R}^{n^k})$, we set
\begin{align}\label{smooth_div_k}
{\operatorname*{div_k}} F = \sum_{\beta \in \mathbb N^n, |\beta|=k} \partial^\beta F_\beta.
\end{align} 
 For $F\in L^1_{\rm loc}(\rn, \R^{n^k})$,  the  equality $\operatorname*{div}_kF=0$ has to be interpreted in the sense of distributions, namely:
\begin{align}\label{div_free_higherorder}
\int_{\mathbb{R}^n} F \cdot \nabla^k \varphi\;dx =0
\end{align}
for every $\varphi \in C^\infty_c(\rn)$.  \\
The   definitions above differ slightly from the convention of Van Schaftingen in \cite{VS3}, a point we now clarify.  The subspace of symmetric functions on $\rn$ with values in $\mathbb{R}^{n^k}$ has dimension
\begin{align}
    \label{N}
    N= {n+k-1 \choose k}.
\end{align}
For smooth functions on $\R^n$ with values in $\R^N$, the formula \eqref{smooth_div_k} is Van Schaftingen's definition of ${\operatorname*{div}_k}$.    With an abuse of notation we utilize this symbol to denote both of these differentiations. 
To mediate between the two, one identifies
%\footnote{If one prefers to lift $F$ to a symmetric tensor, one needs to introduce combinatorial constants in the formula that follows.} 
a function $F$ with values in $\R^N$ and a tensor $\overline{F}$ with values in $\R^{n^k}$ that has $N$ non-zero components:
\begin{align*}
{\operatorname*{div_k}} F = {\operatorname*{div_k}}  \overline{F}.
\end{align*}
Alternatively, if one prefers to lift $F$ to a symmetric tensor, one needs to introduce combinatorial constants in the formula above.
\\
According to the notation introduced above,  $X_{\operatorname*{div}_k}(\rn, \om)$ denotes the subspace of functions in a rearrangement-invariant space $X(\rn, \om)$ which satisfy $\operatorname*{div}_kF=0$ in the sense of distributions.

\begin{theorem} [{\bf Riesz potential estimates in $\rn$}]  \label{sobolev}  Let $n, m, k \in \N$, with $m,n \geq 2$,   and $\alpha \in (0,n)$. Let  $\|\cdot\|_{X(0,\infty)}$ and $\|\cdot\|_{Y(0,\infty)}$ be rearrangement-invariant function norms and   let $\mathcal L$ be any linear homogeneous   $k$-th order  co-canceling operator. Assume that there exists a constant $c_1$ such that
\begin{equation}\label{bound1}
\bigg\|\int_s^\infty r^{-1+\frac \alpha n}f(r)\, dr \bigg\|_{Y(0,\infty)} \leq c_1 \|f\|_{X(0, \infty)}
\end{equation}
for every $f \in X(0,\infty)$. 
Then, there exists a constant $c_2=c_2(c_1, \alpha, \mathcal L)$ such that
\begin{equation}\label{bound2}
\|I_\alpha F\|_{Y(\mathbb R^n, \mathbb R^m)} \leq c_2 \|F\|_{X(\rn, \mathbb R^m)}
\end{equation}
for every  $F \in X_{\mathcal L}(\rn, \mathbb R^m)$.
\end{theorem}

%Our assertion about the connections between the inequalities \eqref{may5} and \eqref{may6} rests upon the fact that the inequality \eqref{bound1}
%is  necessary and sufficient for \eqref{may5} to hold in all  circumstances when the space $V^{\alpha}X(\rn)$
%is well-defined.

As a consequence of Theorem \ref{sobolev}, one can deduce that the inequality \eqref{may6} holds if $Y(\mathbb R^n, \mathbb R^m)$ agrees with the optimal target space  in the Sobolev inequality \eqref{may5}. The vectorial version $X_\alpha(\mathbb R^n, \mathbb R^m)$ of the optimal space in question is defined via its associate   space $X_\alpha'(\mathbb R^n, \mathbb R^m)$.  The latter space is built upon the function norm    obeying
\begin{align}
    \label{may7}
    \|f\|_{X_
    \alpha'(0, \infty)}= \|s^{\frac \alpha n} f^{**}(s)\|_{X'(0,\infty)}
\end{align}
for $f\in \Mpl (0, \infty)$. Here, $\|\,\cdot \,\|_{X'(0,\infty)}$ denotes the function norm which defines the   associate space of $X(\mathbb R^n, \mathbb R^m)$. The right-hand side of \eqref{may7} is a rearrangement-invariant function norm provided that
\begin{equation}\label{may8}
    \|(1+r)^{-1+\frac{\alpha}{n}}\|_{X'(0,\infty)}< \infty,
\end{equation}
see \cite[Theorem 4.4]{EMMP}.   Notice that the condition \eqref{may8} is necessary for an inequality of the form \eqref{bound1} to hold whatever the rearrangement-invariant space $Y(\rn, \rM)$ is. This follows analogously to \cite[Equation (2.2)]{EGP}.
\\
Thanks to \cite[Theorem~4.4]{EMMP}, the inequality \eqref{bound1} holds with $Y(0,\infty)= X_\alpha (0, \infty)$. Hence, the next result can be deduced from an application of Theorem \ref{sobolev}.

\begin{theorem}[{\bf Target space for Riesz potentials  in $\rn$}] \label{sobolev-opt}    Assume that $n,m,k, \alpha$ and $\mathcal L$ are as in Theorem \ref{sobolev}.
Let $\|\cdot\|_{X(0,\infty)}$ be a rearrangement-invariant function norm satisfying \eqref{may8}.
Then, there exists a constant $c=c(\alpha, \mathcal L)$ such that
\begin{equation}\label{bound-opt}
\|I_\alpha F\|_{X_\alpha(\mathbb R^n, \mathbb R^m)} \leq c \|F\|_{X(\rn, \mathbb R^m)}
\end{equation}
for every  $F \in X_{\mathcal L}(\rn, \mathbb R^m)$.
\end{theorem}

%\todo[inline]{A: maybe the following results about inequalities on domains with finite measure can just be presented in a remark}

The following variant of  Theorem \ref{sobolev} holds 
for functions $F$ supported in sets with finite measure.

\begin{corollary}[{\bf Riesz potential estimates in domains}] \label{sobolev-domain}  
  Assume that $n,m,k, \alpha$ and $\mathcal L$ are as in Theorem \ref{sobolev}. Let 
$\Omega$ be a measurable set in $\rn$ such that $|\Omega|<\infty$ and let   $\|\cdot\|_{X(0,|\Omega|)}$ and $\|\cdot\|_{Y(0,|\Omega|)}$ be rearrangement-invariant function norms.
Assume that there exists a constant $c_1$ such that
\begin{equation}\label{bound1-domain}
\bigg\|\int_s^{|\Omega|} r^{-1+\frac \alpha n}f(r)\, dr \bigg\|_{Y(0,|\Omega|)} \leq c_1 \|f\|_{X(0, |\Omega|)}
\end{equation}
for every $f \in X(0,|\Omega|)$. 
Then, there exists a constant $c_2=c_2(c_1, \alpha, \mathcal L)$ such that
\begin{equation}\label{bound2-domain}
\|I_\alpha F\|_{Y(\Omega, \mathbb R^m)} \leq c_2 \|F\|_{X(\Omega, \mathbb R^m)}
\end{equation}
for every  $F \in X_{\mathcal L}(\rn, \mathbb R^m)$ vanishing outside $\Omega$.
\end{corollary}

Accordingly, a counterpart of Theorem \ref{sobolev-opt} tells us that the inequality \eqref{bound2-domain} holds with $Y(\Omega, \rM)=X_\alpha (\Omega, \R^m)$, where $X_\alpha (\Omega, \R^m)$ is the rearrangement invariant space whose associate norm is defined  by the function norm
\begin{align}
    \label{may7-domain}
    \|f\|_{X_
    \alpha'(0, |\Omega|)}= \|s^{\frac \alpha n} f^{**}(s)\|_{X'(0,|\Omega|)}
\end{align}
for $f\in \Mpl (0, |\Omega|)$.
Notice that no additional assumption like \eqref{may8} is required in this case.
\\ Since,
 by \cite[Theorems 2.1 and 2.3]{mihula}, the inequality \eqref{bound1-domain} holds with $Y(0,|\Omega|)= X_\alpha (0, |\Omega|)$, the following result is a consequence of Corollary \ref{sobolev-domain}.

\begin{corollary}[{\bf Target space for Riesz potentials  in domains}] \label{sobolev-opt-domain}  
  Assume that $n,m,k, \alpha$, $\mathcal L$, $\Omega$,  and $\|\cdot\|_{X(0,\infty)}$ are as in Corollary \ref{sobolev-domain}. 
Then, there exists a constant $c=c(\alpha, \mathcal L, |\Omega|)$ such that
\begin{equation}\label{bound-opt-domain}
\|I_\alpha F\|_{X_\alpha(\Omega, \mathbb R^m)} \leq c \|F\|_{X(\Omega, \mathbb R^m)}
\end{equation}
for every  $F \in X_{\mathcal L}(\rn, \mathbb R^m)$ vanishing outside $\Omega$.
\end{corollary}

\begin{remark}
    \label{riesz-classical}
  {\rm If the constraint \eqref{constr}  is dropped, namely all functions $F \in X(\rn, \mathbb R^m)$ are admitted,
   the inequality \eqref{bound2} is known to hold if and only if the assumption \eqref{bound1} is coupled with the additional inequality for the dual Hardy type operator
\begin{equation}\label{bound1dual}
\bigg\| s^{-1+\frac \alpha n}\int_0^sf(r)\, dr \bigg\|_{Y(0,\infty)} \leq c\|f\|_{X(0, \infty)}
\end{equation}
for some constant $c$ and every $f \in X(0,\infty)$. The proof of this assertion follows exactly along the same lines as that of \cite[Theorem 2]{Cianchi-JLMS}, dealing with Orlicz norms. 
\\ A parallel property holds in connection with the inequality \eqref{bound2-domain}, whose validity for every $F\in X(\rn, \rM)$ vanishing outside $\Omega$ is equivalent to the couple of inequalities consisting of \eqref{bound1-domain} and of a counterpart of \eqref{bound1dual} with $X(0, \infty)$ and $X(0, \infty)$ replaced with $X(0,|\Omega|)$ and $Y(0,|\Omega|)$.
\iffalse

\begin{equation}\label{bound1dual}
\bigg\| s^{-1+\frac \alpha n}\int_0^sf(r)\, dr \bigg\|_{Y(0,|\Omega|)} \leq c_1 \|f\|_{X(0,|\Omega|)}
\end{equation}
for $f \in X(0,|\Omega|)$.
\fi
   }
\end{remark}

The  theorems  above can be used  to derive a  number of new inequalities in families of rearrangement-invariant spaces. They include, for instance, Zygmund spaces and, more generally, Orlicz spaces and  Lorentz-Zygmund spaces.

\par We begin with Orlicz domain and target spaces. Given 
  a~Young function $A$ such that
\begin{equation}\label{conv0}
	\int_{0}\left(\frac{t}{A(t)}\right)^{\frac{\alpha}{n-\alpha}}\,dt < \infty,
\end{equation}	
let 
 $A_{\frac{n}{\alpha}}$ be its Sobolev conjugate defined as
\begin{equation}\label{An}
A_{\frac{n}{\alpha}} (t) = A(H^{-1}(t)) \quad \text{for $t\geq 0$,}
\end{equation}
where
\begin{equation}\label{H}
H(t) = \bigg(\int _0^t \bigg(\frac \tau{A(\tau)}\bigg)^{\frac
{\alpha}{n-\alpha}} d\tau\bigg)^{\frac {n-\alpha}n} \quad \text{for $t \geq0$.}
\end{equation}
When $\alpha \in \N$, the Young function $A_{\frac{n}{\alpha}}$  defines the optimal Orlicz target space for embeddings of the $\alpha$-th order Orlicz-Sobolev space $V^\alpha L^A(\rn)$.
This is shown in  \cite{cianchi_CPDE}  for $\alpha =1$, and in \cite{cianchi_forum} for an arbitrary integer $\alpha \in (0,n)$
(see also \cite{cianchi_IUMJ} for an alternate equivalent formulation). An analogous result for fractional-order Orlicz-Sobolev spaces is established in \cite{ACPS}.
From Theorems \ref{sobolev} and \ref{sobolev-domain} one can deduce, via \cite[Inequality (2.7)]{cianchi_CPDE}, that the same target space is admissible for Riesz potential inequalities under the constraint \eqref{constr}.

\begin{theorem}[{\bf Riesz potential  inequalities in Orlicz spaces}]\label{sobolev-orlicz} 
  Let $n,m,k, \alpha$ and $\mathcal L$ be as in Theorem \ref{sobolev}.
\iffalse 
Let $n, m \in \N$, with $m,n \geq 2$, and $\alpha \in (0,n)$, and let $\mathcal L$ be any linear homogeneous first-order  co-canceling operator.
\fi
\\ (i) Assume that $A$ is a Young function fulfilling the condition \eqref{conv0}.
Then, there exists a constant $c=c(\alpha, \mathcal L)$ such that
\begin{equation}\label{bound-orlicz}
\|I_\alpha F\|_{L^{A_{\frac n\alpha}}(\mathbb R^n, \mathbb R^m)} \leq c \|F\|_{L^A(\rn, \mathbb R^m)}
\end{equation}
for   $F \in L^A_{\mathcal L}(\rn, \mathbb R^m)$.
\\ (ii) Assume that $\Omega$ is a measurable set in $\rn$ such that $|\Omega|<\infty$.   Let $A$ be a Young function and let   $A_{\frac n\alpha}$ be a Young function defined as in \eqref{An} with $A$ modified, if necessary, near $0$ in such a way that the condition \eqref{conv0} is fulfilled. Then, there exists a constant $c=c(\alpha, A, \mathcal L, |\Omega|)$ such that
\begin{equation}\label{bound-orlicz-domain}
\|I_\alpha F\|_{L^{A_{\frac n\alpha}}(\Omega, \mathbb R^m)} \leq c \|F\|_{L^A(\Omega, \mathbb R^m)}
\end{equation}
for    $F \in L^A_{\mathcal L}(\Omega, \mathbb R^m)$ vanishing outside $\Omega$.
\end{theorem}

\begin{remark}
    \label{linfinity}
   {\rm  If $A$ grows  so fast   near infinity that
\begin{equation}\label{convinf}
	\int^\infty\left(\frac{t}{A(t)}\right)^{\frac{\alpha}{n-\alpha}}\,dt < \infty,
\end{equation}	
then $A_{\frac n\alpha}(t)=\infty$ for large $t$. Hence, $L^{A_{\frac n\alpha}}(\rn, \mathbb R^m)\to L^\infty (\rn, \mathbb R^m)$ and the inequality \eqref{bound-orlicz}
implies that
\begin{equation}\label{bound-orlicz-inf}
\|I_\alpha F\|_{L^{\infty}(\mathbb R^n, \mathbb R^m)} \leq c \|F\|_{L^A(\rn, \mathbb R^m)}
\end{equation}
for   $F \in L^A_{\mathcal L}(\rn, \mathbb R^m)$.}
\end{remark}

\begin{example}\label{ex1}{\rm
Consider a Young function $A$ such that
\begin{equation}\label{june3}
A(t) \,\,\simeq\,\, \begin{cases} t^{p_0} (\log \frac 1t)^{r_0} & \quad \text{near zero}
\\
t^p  (\log t)^r & \quad \text{near infinity,}
\end{cases}
\end{equation}
where either $p_0>1$ and $r_0 \in \R$, or $p_0=1$ and $r_0 \leq 0$, and either $p>1$ and $r \in \R$, or $p=1$ and $r \geq 0$.  
\\  The function $A$ satisfies the assumption \eqref{conv0} if
\begin{equation}\label{june4}
\text{either $1\leq p_0< \frac n\alpha$ and $r_0$ is as above, or $p_0=\frac n\alpha$ and $r_0 > \frac n\alpha -1$.}
\end{equation}
\\ Theorem \ref{sobolev-orlicz}  tells us that the inequality \eqref{bound-orlicz} holds,  where
\begin{equation}\label{june5}
A_{\frac n\alpha}(t)\,\,  \simeq\,\, \begin{cases} t^{\frac {n{p_0}}{n-\alpha{p_0}}} (\log \frac 1t)^{\frac {nr_0}{n-\alpha{p_0}}} & \quad \text{ if $1\leq {p_0}< \frac n\alpha$ }
\\
e^{-t^{-\frac{n}{\alpha(r_0 +1)-n}}} & \quad  \text{if ${p_0}=\frac n\alpha$ and $r_0 > \frac n\alpha -1$}
\end{cases} \quad \text{near zero,}
\end{equation}
and
\begin{equation}\label{dec256}
A_{\frac n\alpha}(t) \,\, \simeq \,\, \begin{cases} t^{\frac {np}{n-\alpha p}} (\log t)^{{\color{black}\frac {n r  }{n-\alpha p}}} & \quad  \text{ if $1\leq p< \frac n\alpha$ }
\\
e^{t^{\frac{n}{n-(r +1)\alpha}}}&  \quad \text{if  $p=\frac n\alpha$ and $r < \frac n\alpha -1$}
\\
e^{e^{t^{\frac n{n-\alpha}}}} &  \quad \text{if  $p=\frac n\alpha$ and $r = \frac n\alpha -1$}
\\
\infty &  \quad \text{otherwise}
\end{cases} \quad \text{near infinity.}
\end{equation} 
In particular, the choice $p_0=p=1$ and $r_0=0$ yields the inequality \eqref{exlog}.   By contrast, as noticed in Section \ref{intro}, this inequality fails if the constraint \eqref{constr} is dropped. This claim can be verified via an application of \cite[Theorem 2]{Cianchi-JLMS}, where boundedness properties of Riesz potentials in Orlicz spaces are characterized.
\\ Analogous conclusions hold with regard to
the inequality \eqref{bound-orlicz-domain}. However,  since  $|\Omega|<\infty$,
only the behaviours near infinity of $A$ and $A_{\frac n\alpha}$ displayed above are relevant in this case. In particular, the assumption \eqref{june4} can be dropped.

}\end{example}

\begin{example}\label{ex2}{\rm
Let $A$ be a Young function such that
\begin{equation}\label{sep50}
A(t) \,\,\simeq\,\, \begin{cases} t^{p_0} (\log (\log \frac 1t))^{r_0} & \quad \text{near zero}
\\
t^p  (\log (\log t))^r & \quad \text{near infinity,}
\end{cases}
\end{equation}
where either $p_0>1$ and $r_0 \in \R$, or $p_0=1$ and $r_0 \leq 0$, and either $p>1$ and $r \in \R$ or $p=1$ and $r \geq 0$.  
\\  This function satisfies the assumption \eqref{conv0} if
\begin{equation}\label{sep51}
\text{$1\leq p_0< \frac n\alpha$ and $r_0$ is as above.}
\end{equation}
\\ From Theorem \ref{sobolev-orlicz}  we infer that the inequality \eqref{bound-orlicz} holds,  with
\begin{equation}\label{sep52}
A_{\frac n\alpha}(t)\,  \simeq\, t^{\frac {n{p_0}}{n-\alpha{p_0}}} (\log (\log \tfrac 1t))^{\frac {nr_0}{n-\alpha{p_0}}} \qquad \text{near zero,}
\end{equation}
and
\begin{equation}\label{sep53}
A_{\frac n\alpha}(t) \,\, \simeq \,\, \begin{cases} t^{\frac {np}{n-\alpha p}} (\log (\log t))^{{\color{black}\frac {n r  }{n-\alpha p}}} & \quad  \text{ if $1\leq p< \frac n\alpha$ }
\\ \\
e^{t^{\frac{n}{n-\alpha}}(\log t)^{\frac {r\alpha}{n-\alpha}}} &  \quad \,\,\text{if  $p=\frac n\alpha$}
\end{cases} \quad \text{near infinity.}
\end{equation} 
For $p_0=p=1$ and $r_0=0$ this results in  the inequality \eqref{exloglog}. 
 The failure of such inequality without 
 the constraint \eqref{constr}  can be demonstrated by  \cite[Theorem 2]{Cianchi-JLMS}.
\\ Conclusions in the same spirit hold for the inequality \eqref{bound-orlicz-domain}, where $|\Omega|<\infty$, with  simplifications analogous to those described in Example \ref{ex1}.}
\end{example}

The inequality \eqref{bound-orlicz}
can be improved if norms of Orlicz-Lorentz type are allowed on its left-hand side. 
\\ Let $A$ be  a Young function fulfilling the condition \eqref{conv0} and let $a: [0, \infty) \to [0, \infty)$ be the left-continuous function such that
\begin{align}
    \label{a}
    A(t) = \int_0^ta(\tau)\, d\tau \qquad \text{for $t\geq 0$.}
\end{align}
Denote by  $\widehat A$ the Young function given by
\begin{equation}\label{E:1}
	\widehat A (t)=\int_0^t\widehat a (\tau)\,d\tau\quad\text{for $t\geq 0$},
\end{equation}
where
\begin{equation}\label{E:2}
	{\widehat a\,}^{-1}(r) = \left(\int_{a^{-1}(r)}^{\infty}
		\left(\int_0^t\left(\frac{1}{a(\varrho)}\right)^{\frac{\alpha}{n-\alpha}}\,d\varrho\right)^{-\frac{n}{\alpha}}\frac{dt}{a(t)^{\frac{n}{n-\alpha}}}
				\right)^{\frac{\alpha}{\alpha-n}}
					\quad\text{for $r\ge0$}.
\end{equation}
Let $L(\widehat A,\frac{n}{\alpha})(\rn, \rM)$ be the  Orlicz-Lorentz space  defined as in \eqref{sep35}. Namely, $L(\widehat A,\frac{n}{\alpha})(\rn, \rM)$ is the rearrangement-invariant space associated with the function norm given by
\begin{equation}\label{E:29}
	\|f\|_{L(\widehat A,\frac{n}{\alpha})(0,\infty)}
		= \|r^{-\frac{\alpha}{n}}f^{*}(r)\|_{L^{\widehat A}(0, {\infty})}
\end{equation}
for $f \in \Mpl (0, \infty)$.
\\ The conclusions of our result about Riesz potential inequalities with Orlicz-Lorentz target spaces are best stated by distinguishing into the cases when the function $A$ fulfils \eqref{convinf} or the complementary condition
\begin{equation}\label{divinf}
	\int^\infty\left(\frac{t}{A(t)}\right)^{\frac{\alpha}{n-\alpha}}\,dt =\infty.
\end{equation}	
They are the subject of the following theorem,  which is a consequence of Theorem \ref{sobolev} and \cite[Inequalities (3.1) and (3.2)]{cianchi_ibero}.

\begin{theorem} [{\bf Riesz potential inequalities with Orlicz-Lorentz  targets}]\label{sobolev-orliczlorentz}
 Let $n,m,k, \alpha$ and $\mathcal L$ be as in Theorem \ref{sobolev}.
\iffalse
Let $n, m \in \N$, with $m,n \geq 2$, and $\alpha \in (0,n)$, and let $\mathcal L$ be any linear homogeneous first-order  co-canceling operator.
\fi
 Let $A$ be a Young function fulfilling the condition \eqref{conv0}.
\\ (i) Assume that \eqref{divinf} holds.
Then, there exists a constant $c=c(\alpha, \mathcal L)$ such that
\begin{equation}\label{bound-orliczlorentz1}
\|I_\alpha F\|_{L(\widehat A,\frac{n}{\alpha})(\rn, \rM)} \leq c \|F\|_{L^A(\rn, \mathbb R^m)}
\end{equation}
for    $F \in L^A_{\mathcal L}(\rn, \mathbb R^m)$.
\\ (ii) Assume that \eqref{convinf} holds. Then, 
there exists a~constant   $c=c(\alpha, A, \mathcal L)$ such that
\begin{equation}\label{bound-orliczlorentz2}
	\|I_\alpha F\|_{(L^\infty \cap L(\widehat A,\frac{n}{\alpha}))(\rn, \rM)} \leq c \|F\|_{L^A(\rn, \mathbb R^m)}
\end{equation}
for    $F \in L^A_{\mathcal L}(\rn, \mathbb R^m)$.
\end{theorem}

\begin{remark}\label{finite measure} {\rm A version of Theorem \ref{sobolev-orliczlorentz} holds for functions vanishing outside a set $\Omega$ of finite measure. In the norms in \eqref{bound-orliczlorentz1} and \eqref{bound-orliczlorentz2}, the set $\rn$ has to be replaced with $\Omega$, and  the condition \eqref{conv0} can be disregarded in this case. Moreover, the space $(L^\infty \cap L(\widehat A,\frac{n}{\alpha}))(\Omega, \rM)$ in \eqref{bound-orliczlorentz2} agrees with $L^\infty (\Omega, \rM)$, up to equivalent norms.}
\end{remark}

\begin{example}\label{ex orlicz-lorentz}{\rm
Consider a Young function $A$ as in \eqref{june3}--\eqref{june4}. Theorem \ref{sobolev-orliczlorentz}
 tells us that, if 
\begin{equation}\label{june7}
\text{either $1\leq p< \frac n\alpha$,  or $p=\frac n\alpha$ and $r \leq \frac n\alpha -1$,}
\end{equation}
then the inequality \eqref{bound-orliczlorentz1} holds with
\begin{equation}\label{june8}
\widehat A(t)\,\,  \simeq\,\, \begin{cases} t^{p_0}(\log \frac 1t)^{r_0} & \quad \text{ if $1\leq {p_0}< \frac n\alpha$ }
\\
t^{\frac n\alpha}  (\log \frac 1t)^{r_0 - \frac n\alpha} & \quad  \text{if ${p_0}=\frac n\alpha$ and $r_0> \frac n\alpha -1$}
\end{cases} \quad \text{near zero,}
\end{equation}
and
\begin{equation}\label{dec254}
\widehat A(t) \,\, \simeq \,\, \begin{cases} t^p (\log t)^r & \quad  \text{ if $1\leq p< \frac n\alpha$ }
\\
t^{\frac n\alpha}  (\log t)^{r - \frac n\alpha} &  \quad \text{if  $p=\frac n\alpha$ and $r < \frac n\alpha -1$}
\\
t^{\frac n\alpha}  (\log t)^{-1} (\log (\log t))^{-\frac n\alpha}&  \quad \text{if  $p=\frac n\alpha$ and $r = \frac n\alpha -1$}
\end{cases} \quad \text{near infinity.}
\end{equation}
In particular, the   choice ${p_0}=p<\tfrac n\alpha$ and $r_0 =r=0$ yields $\widehat A(t) =t^p$.
\\ From an application of \cite[Lemma 6.12, Chapter 4]{BennettSharpley}, one can deduce that, if $1\leq p=p_0<\frac n\alpha$ and $r=0$, then 
\begin{align}\label{sep40}
    L(\widehat A,\tfrac{n}{\alpha})(\rn, \rM)= L^{\frac {np}{n-\alpha p},p}(\log L)^{\frac r p}(\rn, \rM),
\end{align}
up to equivalent norms. Hence, the inequality \eqref{exLZ} follows 
by choosing $p_0=p=1$ and $r_0=0$. Characterizations of the space $L(\widehat A,\tfrac{n}{\alpha})(\rn, \rM)$, analogous to \eqref{sep40}, for ${p_0}=p=\tfrac n\alpha$, in terms of 
 Lorentz-Zygmund or generalized  Lorentz-Zygmund spaces are also available -- see e.g. \cite[Example 1.2]{cianchi_ibero}.
\\   As claimed in Section \ref{intro}, the inequality \eqref{exLZ}  
breaks down in the space of all functions $F\in L (\log L)^r(\rn, \rM)$. Indeed, in  the light of Remark \ref{riesz-classical}, this inequality without the constraint \eqref{constr} would imply that
\begin{equation}\label{boundOL-fail}
\bigg\| s^{-1+\frac \alpha n}\int_0^sf(r)\, dr \bigg\|_{L^{\frac n{n-\alpha},1,r}(0,\infty)} \leq c \|f\|_{ L (\log L)^r(0,\infty)}
\end{equation}
for some constant $c$ and every $f \in L (\log L)^r(0,\infty)$. Thanks to \eqref{eqZYg}, the latter inequality in turn implies that
\begin{align}
    \label{OL-fail}  \int_0^1\bigg(
    (\cdot)^{-1+\frac \alpha n}\int_0^{(\cdot)}f(r)\, dr
   \bigg)^{**}(s) s^{\frac{n-\alpha}n}(\log_+ \tfrac 2s\big)^{r}\frac{ds}s \leq c \int_0^1 f^{*}(s) (\log_+ \tfrac 2s\big)^{r} \, ds
\end{align}
for all $f\in \Mpl (0,1)$ making the right-hand side finite. Consider functions
$f$ of the form 
\begin{align}
    \label{f}
    f(s)= \tfrac 1s \big(\log \tfrac 1s\big)^{-\gamma},
\end{align}
 with 
   $1+r  < \gamma < 2+r$. Then, 
\begin{align}
    \label{f**}
    f^{*}(s) \approx \tfrac 1s \big(\log \tfrac 1s\big)^{-\gamma} \quad \text{and} \quad \bigg(
    (\cdot)^{-1+\frac \alpha n}\int_0^{(\cdot)}f(r)\, dr
   \bigg)^{**}(s) \approx s^{\frac \alpha n -1} \big(\log \tfrac 1s\big)^{1-\gamma},
\end{align}
up to multiplicative constants independent of $s\in (0,1)$.  Under 
our  assumptions on $\gamma$,  the right-hand side of \eqref{OL-fail} is finite, whereas its left-hand side is infinite. This demonstrates that
 the inequality \eqref{OL-fail} fails. 
\\ According to Remark \ref{finite measure},  inequalities parallel to \eqref{bound-orliczlorentz1} and \eqref{bound-orliczlorentz2} for functions $F$ supported in sets $\Omega$, with $|\Omega|<\infty$, hold even if $A$ does not satisfy the assumption \eqref{june4}.  The only relevant piece of information is indeed the behaviour near $\infty$ of $A$ and $\widehat A$ described in \eqref{june3} and \eqref{dec254}.

}
\end{example}

\begin{example}\label{lorentzzygmund} {\rm We conclude with an application of Corollary \ref{sobolev-opt-domain} to Lorentz-Zygmund spaces. For brevity, we limit ourselves to domain spaces whose  first index equals $1$, namely to spaces of the form $L^{(1,q,r)}(\Omega, \rM)$, with $q\in [1, \infty)$.
 As explained in Section \ref{intro}, these are the most relevant in the present setting, since the Riesz potential inequality with the same target space fails if the co-canceling constraint is dropped. 
In order to avoid introducing new classes of functions spaces, we also assume, for simplicity, that $r>-\frac 1q$. 
\\ Assume that $\Omega$ is a measurable set in $\rn$ such that $|\Omega|<\infty$. An application of Corollary \ref{sobolev-opt-domain}, combined with a result of \cite{cavaliere-new} where  an estimate for the norm \eqref{may7-domain} is determined for the space 
$L^{(1,q,r)}(\Omega, \rM)$, tells us that 
    \begin{equation*}
 \|I_\alpha F\|_{L^{\frac n{n-\alpha},q,r+1}(\Omega, \rM)} \leq c \|F\|_{L^{(1,q,r)}(\Omega, \rM)}
\end{equation*}
for some constant $c$ and   every $F \in L^{(1,q,r)}_{\mathcal L}(\rn, \mathbb R^m)$ vanishing outside $\Omega$. The restriction to sets $\Omega$ with finite measure is needed for  the space $L^{(1,q,r)}(\Omega, \mathbb R^m)$ not to be trivial.
\\  On the other hand, the inequality in question does not hold for functions which do not satisfy the co-canceling condition \eqref{constr}. Actually, by Remark \ref{riesz-classical}, if such an inequality were true, then we would have 
\begin{equation}\label{boundLZ-fail}
\bigg\| s^{-1+\frac \alpha n}\int_0^sf(r)\, dr \bigg\|_{L^{\frac n{n-\alpha},q,r+1}(0,|\Omega|)} \leq c_1 \|f\|_{L^{(1,q,r)}(0,|\Omega|)}
\end{equation}
for some constant $c$ and every $f \in L^{(1,q,r)}(0,|\Omega|)$. By assuming, without loss of generality, that $|\Omega|=1$, the inequality \eqref{boundLZ-fail} reads
\begin{align}
    \label{LZ-fail}   \int_0^1\Bigg[\bigg(
    (\cdot)^{-1+\frac \alpha n}\int_0^{(\cdot)}f(r)\, dr
   \bigg)^{**}(s) s^{\frac{n-\alpha}n}(\log \tfrac 2s\big)^{r+1}\Bigg]^q\frac{ds}s \leq c \int_0^1\Big(f^{**}(s) s(\log \tfrac 2s\big)^{r}\Big)^q\frac{ds}s
\end{align}
for all $f\in \Mpl (0,1)$ making the right-hand side finite. Such an inequality fails for any function $f$ as in  \eqref{f}, with 
$1+r+\tfrac 1q < \gamma < 2+r+\tfrac 1q$. Indeed,  owing to equation \eqref{f**}, one can verify that the right-hand side of \eqref{LZ-fail} is finite, whereas its left-hand side is infinite.}
\end{example}

 \section{A rearrangement estimate}\label{estimate}

As mentioned above, a crucial ingredient in the proof of Theorem \ref{sobolev} is the rearrangement estimate for Riesz potentials of $k$-th order divergence free vector fields provided by the following theorem.

 \begin{theorem}\label{K-CZ}
  Let $k \in \N$, $n, l, \ell \geq 2$, and let $N$ be as in \eqref{N}. Let $\alpha \in (0,n)$.  
%\begin{align*}
  %  T&:  L^1(\rn,\om) \to L^{1,\infty}(\rn,\mathbb{R}^\ell)\\
  %  T&:L^{p,q}(\rn,\om) \to  L^{p,q}(\rn,\mathbb{R}^\ell)
%\end{align*}
%for $1<p<+\infty$, $1\leq q \leq +\infty$
Then, there exists a positive constant $c=c(\alpha,n,k)$ such that
\begin{align}\label{nov100}
\int_0^t  s^{-\frac{\alpha}{n}} (I_\alpha F)^*(s)\, ds
\leq c\int_0^{t}  s^{-\frac{\alpha}{n}} \int_{s}^{\infty} F^*(r)r^{-1+\alpha/n}\,drds  \quad\text{for $t>0$,}
\end{align}
for every 
%\todo[inline]{A: we shall apply the inequality \eqref{nov100} to functions $F$ which are not required to belong to $L^1(\rn, \mathbb{R}^N)+ L^{\frac n\alpha, 1}(\rn, \mathbb{R}^N)$. For these functions the right-hand side is infinite and the inequality trivially holds. What do we have to assume on $F$  for $I_\alpha T Q  F$ to be defined? }
$F \in  L^1(\rn, \mathbb{R}^{N\times l})+ L^{\frac n\alpha, 1}(\rn, \mathbb{R}^{N\times l})$, such that ${\rm div_k}(F^\beta)_i=0$ for $i =1, \ldots,l$.  Here,   $F=[F^\beta]$ with rows $(F^\beta)_i$ for $i = 1,\ldots,l$.
 \end{theorem}

\begin{remark}
    \label{rem-Kfunct}
    {\rm The 
    inequality \eqref{nov100} is  equivalent to the $K$-functional inequality
\begin{equation}\label{dec32second}
K\big(I_\alpha F, t;   L^{\frac{n}{n-\alpha}, 1}(\rn, \R^{N\times l}),  L^\infty(\rn, \R^{N\times l})\big)
\leq c K\big(F, t/c;   L^1_{\operatorname*{div_k}}(\rn, \R^{{N\times }l}),   L^{\frac n\alpha, 1}_{\operatorname*{div_k}}(\rn, \R^{{N\times }l})\big)
\end{equation}
for every $F \in  L^1(\rn, \mathbb{R}^{N\times l})+ L^{\frac n\alpha, 1}(\rn, \mathbb{R}^{N\times l})$, such that 
${\rm div_k}(F^\beta)_i=0$.}
\end{remark}

Theorem \ref{K-CZ} is a special case of \cite[Theorem 5.1]{BCS_canceling}, which, loosely speaking, deals with the Riesz potential operator $I_\alpha$ possibly composed with singular integral operators satisfying customary assumptions. 
In this section we present 
a direct proof of Theorem 
\ref{K-CZ} in the case when $k=1$. 
The critical step is a formula for the $K$-functional of divergence  the couple $(L^1_{\operatorname*{div}}(\rn, \rn),L^{p,q}_{\operatorname*{div}}(\rn, \rn))$.   This is the content of the next result.

\begin{theorem}[{\bf $K$-functional for divergence--free vector fields}]

\label{lemma2}
Let $p \in (1,\infty)$ and $q\in [1, \infty]$.  Then,
\begin{align}\label{sep1}
K(F,t,L^1_{\operatorname*{div}}(\rn, \rn),L^{p,q}_{\operatorname*{div}}(\rn, \rn)) & \approx K(F,t,L^1(\rn, \rn),L^{p,q}(\rn, \rn))
\\ \nonumber & \approx
\int_0^{t^{p'}}F^*(s)\, ds + t \bigg(\int_{t^{p'}}^\infty s^{-1+\frac qp} F^*(s)\, ds\bigg)^{\frac 1q}  \qquad \text{for $t>0$,}
\end{align}
for every   $F \in L^1 (\rn, \rn) +L^{p,q}(\rn, \rn)$ such that $\operatorname*{div}F=0$, with equivalence constants depending on $n, p, q$. 
\end{theorem}

The proof of Theorem \ref{lemma2} builds upon results from \cite{Bourgain,Pisier}. It requires a precise analysis of   mapping properties of the  Helmholz projection singular integral operator under the constraint $\operatorname*{div}F=0$. This is the content of 
Lemma \ref{lemma1} below.
 The relevant operator  is formally defined as 
   \begin{align}
    \label{aug17}\mathcal HF = \nabla \operatorname*{div} (-\Delta)^{-1} F
\end{align}
 for   $F \in L^1(\rn, \rn) +L^{p,q}(\rn, \rn)$.
 Observe that, owing to Fourier calculus,  
\begin{align*}
\mathcal H \Phi = \left( -\frac{\xi}{|\xi|} \frac{\xi}{|\xi|}\cdot \widehat{\Phi}(\xi) \right)\widecheck{\phantom{x}}
\end{align*}
 if $\Phi \in C^\infty_c(\mathbb{R}^n,\mathbb{R}^n)$.   
 \\ The operator $\mathcal H$ is bounded on $L^p(\mathbb{R}^n,\mathbb{R}^n)$ for $1<p<\infty$ and therefore also on $L^{p,q}(\mathbb{R}^n,\mathbb{R}^n)$ for $1<p<\infty$ and $1\leq q \leq \infty$.  This can be seen, for instance, because 
\begin{align*}
\left|\frac{\partial^\beta}{\partial \xi^\beta} \frac{\xi}{|\xi|} \frac{\xi}{|\xi|} \right| \lesssim |\xi|^{-|\beta|} \qquad \text{for $\xi \neq 0$,}
\end{align*}
for every multi-index $\beta \in \mathbb{N}_0^n$.
This allows one to invoke Mihlin's multiplier theorem \cite[Theorem 6.2.7 on p.~446]{grafakos} to prove   boundedness in $L^p$, followed by interpolation \cite[Theorem 1.4.19 on p.~61]{grafakos} for $L^{p,q}$ boundedness.
\\ 
The kernel $\kappa : \rn \setminus \{0\} \to \R$ of the operator $\mathcal H$ fulfills the so called H\"ormander condition:
\begin{align}
    \label{hormander}
    \sup_{x\neq 0}\int_{\{|y|\geq 2|x|\}} |\kappa (y-x)- \kappa (y)|\, dy<\infty,
\end{align}
see \cite[Proof of Theorem 6.2.7]{grafakos}.
\\ Notice that, if $\varphi \in C^\infty_c(\mathbb{R}^n)$, then   
\begin{align}
    \label{aug18}
    \mathcal H \,\nabla \varphi = -\nabla \varphi.
\end{align}

\begin{lemma}
%[{\bf A divergence-free projection operator}]
\label{lemma1}
Let  $p\in (1,\infty)$ and $q\in [1, \infty]$.   Define the operator $P$ as
\begin{align*}
PF = F+\mathcal H F
\end{align*}
for  $F \in L^1(\rn, \rn) +L^{p,q}(\rn, \rn)$. 
\\ (i) If $\operatorname*{div}F=0$ in the sense of distributions, then
\begin{align}\label{aug20}
PF=F.
\end{align}
 (ii)
If $F, PF \in L^1(\rn, \rn)$ or $F, PF \in L^{p,q}(\rn, \rn)$, then  
\begin{align}\label{aug21}
\operatorname*{div} PF = 0
\end{align}
in the sense of distributions.
\end{lemma}

\begin{proof} Throughout this proof, the constants in the relations $\lq\lq \lesssim"$ and $\lq\lq \approx"$ only depend on $n, p, q$.
\\
Part (i).
Let $F \in L^1(\rn, \rn) + L^{p,q}(\rn, \rn)$ be such that $\operatorname*{div}F=0$
 in the sense of distributions, i.e. 
\begin{align}\label{div_free}
\int_{\mathbb{R}^n} F \cdot \nabla \varphi\;dx =0
\end{align}
for all $\varphi \in C^\infty_c(\mathbb{R}^n)$.  We begin by showing that   the identity \eqref{div_free} also holds for all $\varphi \in C^1(\mathbb{R}^n)$ such that $\varphi \in L^\infty(\mathbb{R}^n)\cap L^{n'p',n'q'}(\mathbb{R}^n)$ and $\nabla \varphi \in L^\infty(\rn, \rn)\cap L^{p',q'}(\rn, \rn)$.  To see this, given  such a function $\varphi$  and $R>0$, consider a sequence 
$\{\rho_h\}$  of standard mollifiers supported in the ball $B_{1/h}$,  and  a cutoff function $\eta \in C^1_c(\mathbb{R}^n)$ such that $\eta =1$ on $B_R$, $\eta =0$ outside $B(2R)$ and $|\nabla\eta|\lesssim 1/R$. We have that
\begin{align}\label{aug15}
\int_{\mathbb{R}^n} F \cdot \nabla \varphi\;dx &= 
\lim_{h \to \infty} \bigg(\int_{\mathbb{R}^n} F \cdot \nabla ((\varphi\ast \rho_h) \eta)\;dx +\int_{\mathbb{R}^n} F \cdot \nabla ((\varphi\ast \rho_h)(1-\eta))\;dx\bigg) \\
\nonumber &= \lim_{h \to \infty} \int_{\mathbb{R}^n} F \cdot \nabla ((\varphi\ast \rho_h)(1-\eta))\;dx,
\end{align}
%\todo[inline]{In the next section we use $l$ for the mollifiers because the order of the operator is $k$.  Maybe we should switch the $k$ to $l$ here.}
%\todo[inline]{A: $\ell$ was also taken for the dimension of some target spaces. I switched to $h$ for mollifiers everywhere}
where the first equality holds thanks to the dominated convergence theorem, since $\nabla \varphi \in L^\infty(\rn, \rn)$, and the second one by \eqref{div_free}.
\\ Let $F_1 \in L^1(\rn, \rn)$ and $F_{p,q} \in L^{p,q}(\rn, \rn)$ be such that $F = F_1 + F_{p,q}$.
Thus,
\begin{align}\label{aug8}
\left|\int_{\mathbb{R}^n} F \cdot \nabla ((\varphi\ast \rho_h)(1-\eta))\;dx\right| &\lesssim \int_{\rn \setminus B_R} (|F_1| +|F_{p,q}|)(|\nabla \varphi|* \rho_h)  \;dx \\ \nonumber
&\quad+ \frac 1R \int_{B_{2R} \setminus B_R} (|F_1| +|F_{p,q}|)|(|\varphi| *  \rho_h)\;dx
\end{align} 
for every $h\in \mathbb{N}$ and $R>0$.
%\todo[inline]{A: I still do not see how to get rid of $\rho_h$ in the argument. I am writing a proof keeping the convolution in the inequality above. If you have a simpler proof avoiding this, please replace it. The old proof is commented in the LaTeX file. }
 One has that
\begin{multline}
    \label{aug10}
\int_{\rn \setminus B_R} (|F_1| +|F_{p,q}|)(|\nabla \varphi|* \rho_h)  \;dx 
\\  
\lesssim \|F_1\|_{L^1(\rn \setminus B_R,\rn)} \|\nabla \varphi\|_{L^\infty(\rn, \rn)} + \|F_{p,q}\|_{L^{p,q}(\rn \setminus B_R, \rn)} \|(|\nabla \varphi|\ast\rho_h)\chi_{\rn \setminus B_R}\|_{L^{p',q'}(\rn)}.
\end{multline}
The first addend on the right hand side of the inequality \eqref{aug10} tends to $0$ as $R\to \infty$, uniformly in $h$. As for the second one, recall that 
the convolution operator with kernel $\rho_h$ is bounded in $L^1(\rn)$ and $L^\infty (\rn)$, with norm not exceeding $1$. 
By an interpolation theorem of Calder\'on \cite[Theorem 2.12, Chapter 3]{BennettSharpley},
%\todo{add reference} 
%\todo[inline]{A: done}
it is also bounded in any rearrangement-invariant space, with norm independent of $h$. Thus,
$$\|F_{p,q}\|_{L^{p,q}(\rn \setminus B_R, \rn)} \|(|\nabla \varphi|\ast\rho_h)\chi_{\rn \setminus B_R}\|_{L^{p',q'}(\rn)} \lesssim \|F_{p,q}\|_{L^{p,q}(\rn \setminus B_R,\rn)}.$$
If $q<\infty$, the norm in the space $L^{p,q}(\rn, \rn)$ is absolutely continuous, and hence the latter inequality tells us that also the second addend on the right-hand side of \eqref{aug10} tends to $0$ as $R\to \infty$, uniformly in $h$.
Assume next that $q=\infty$, whence $q'=1$. We claim that
\begin{align}
    \label{aug12}
    \|(|\nabla \varphi|\ast\rho_h)\chi_{\rn \setminus B_R}\|_{L^{p',1}(\rn)} \lesssim \||\nabla \varphi| \chi_{\rn \setminus B_{R-1}}\|_{L^{p',1}(\rn)}
\end{align}
for $h\in \N$ and $R>1$. To verify this claim, fix any measurable set $E\subset \rn$ and observe that, since ${\rm \supp} \rho_h \subset B_1$ for $h\in \N$, an application of Fubini's theorem and the inequality 
$$\chi_{B_1}(y)\chi_{B_{2R}\setminus B_R}(x) \leq \chi_{B_{2R+1}\setminus B_{R-1}}(x-y) \quad \text{for $x,y \in \rn$}$$
imply that
\begin{align}
    \label{aug13}
    \int_E (|\nabla \varphi|\ast\rho_h)\chi_{\rn \setminus B_R}\, dx \leq \int_E (|\nabla \varphi|\chi_{\rn \setminus B_{R-1}})\ast \rho_h\, dx.
\end{align}
Hence, via equation \eqref{aug11} and the inequality \eqref{nov1}, one deduces that
\begin{align}
    \label{aug14}
    (|\nabla \varphi|\ast\rho_h)\chi_{\rn \setminus B_R})^{**} (s) \leq (|\nabla \varphi|\chi_{\rn \setminus B_{R-1}})^{**}(s) \quad \text{for $s>0$.}
\end{align}
The inequality \eqref{aug12} follows from \eqref{aug14}, via \eqref{hardy}.
\\ Thanks to \eqref{aug12},
\begin{align*}
    \|F_{p,q}\|_{L^{p,\infty}(\rn \setminus B_R,\rn)}& \|(|\nabla \varphi|\ast\rho_h)\chi_{\rn \setminus B_R}\|_{L^{p',1}(\rn)} \\ &\lesssim \|F_{p,q}\|_{L^{p,\infty}(\rn \setminus B_R,\rn)}  \||\nabla \varphi| \chi_{\rn \setminus B_{R-1}}\|_{L^{p',1}(\rn,\rn)}.
\end{align*}
The latter inequality and the absolute continuity of the norm in $L^{p',1}(\rn, \rn)$ imply that the second addend on the right-hand side of \eqref{aug10} tends to $0$ as $R\to \infty$, uniformly in $h$, also in this case. Thus, we have shown that the first term on the right-hand side of the inequality \eqref{aug8} is arbitrarily small, uniformly in $h$, provided that $R$ is large enough. 
\\ As far as the second term is concerned, we have that
\begin{align}
    \label{aug16}
    \frac 1R&  \int_{B_{2R} \setminus B_R} (|F_1| +|F_{p,q}|)|(|\varphi| *  \rho_h)\;dx
     \\ \nonumber & \lesssim \frac 1R\|F_{p,q}\|_{L^{p,q}( B_{2R} \setminus B_R)} \|(\varphi\ast\rho_h)\chi_{B_{2R} \setminus B_R}\|_{L^{p',q'}(\rn)}
     \\\nonumber& \quad +\frac 1R \|F_1\|_{L^1(B_{2R} \setminus B_R,\rn)}\|\varphi\|_{L^\infty(\rn)}
       \\ \nonumber & \lesssim  \|F_{p,q}\|_{L^{p,q}( B_{2R} \setminus B_R,\rn)} \|(\varphi\ast\rho_h)\chi_{B_{2R} \setminus B_R}\|_{L^{n'p',n'q'}(\rn)}\\
       & \quad +\frac 1R \|F_1\|_{L^1(B_{2R} \setminus B_R, \rn)}\|\varphi\|_{L^\infty(\rn)}.\nonumber
\end{align}
  Analogously to \eqref{aug14}, one has that
\begin{align}
    \label{aug14bis}
    (|\nabla \varphi|\ast\rho_h)\chi_{B_{2R} \setminus B_R})^{**} (s) \leq (|\nabla \varphi|\chi_{B_{3R} \setminus B_{R/2}})^{**}(s) \quad \text{for $s>0$,}
\end{align}
provided that $R>2$. Therefore,
the terms on the rightmost side of the inequality \eqref{aug16} can be treated similarly to those on the right-hand side of \eqref{aug10}.
 Altogether, thanks to the arbitrariness of $R$, equation \eqref{div_free} follows from   \eqref{aug15} and \eqref{aug8}.
\iffalse
However, for $k\in \mathbb{N}$ and $R>0$ sufficiently large we have, by a change of variables and enlarging the domain of integration,
\begin{align}\label{aug8}
\left|\int_{\mathbb{R}^n} F \cdot \nabla ((\varphi\ast \rho_h)(1-\eta))\;dx\right| &\leq \int_{B_{R/2}^c} (|F_1| +|F_{p,q}|)|\nabla \varphi|  \;dx \\ \nonumber
&\;+ \int_{B_{3R} \setminus B_{R/2}} (|F_1| +|F_{p,q}|)|\frac{|\varphi|}{R}\;dx
\end{align}
where $F = F_1 + F_{p,q}$ for some $F_1 \in L^1(\mathbb{R}^n)$ and $F_{p,q} \in L^{p,q}(\mathbb{R}^n)$. The limit for the first term on the right-hand-side tends to zero as $R \to \infty$ from the assumption $\nabla \varphi \in L^\infty(\mathbb{R}^n)\cap L^{p',q'}(\mathbb{R}^n)$ and Lebesgue's dominated convergence theorem. 

For the second term, the limit involving $|F_1|$ tends to zero by the same argument, using $\varphi \in L^\infty(\mathbb{R}^n)$.  For the remaining portion of the second term, by H\"older's inequality we have
\begin{align*}
\int_{B_{3R}\setminus B_{R/2}}|F_{p,q}| \frac{|\varphi|}{R}\;dx &\lesssim \frac{1}{R}\|F_{p,q}\|_{L^{p,q}(B_{R/2}^c)} \|\varphi\|_{L^{p',q'}(B_{2R}\setminus B_{R})}\\
&\lesssim \|F_{p,q}\|_{L^{p,q}(B_{R/2}^c)} \|\varphi\|_{L^{n'p',n'q'}(B_{3R}\setminus B_{R/2})}.
\end{align*}
As $\varphi \in L^{n'p',n'q'}(\mathbb{R}^n)$ and $F_{p,q} \in L^{p,q}(\mathbb{R}^n)$, the right-hand-side is finite, while the limit tends to zero as $R \to \infty$ by Lebesgue's dominated convergence theorem applied to the norm on either $L^{n'p',n'q'}(\mathbb{R}^n)$ or $L^{p,q}(\mathbb{R}^n)$, as either $q<+\infty$ or $q'<+\infty$.
\fi
\\ Now, 
for $\Phi \in C^\infty_c(\mathbb{R}^n,\mathbb{R}^n)$, set $\varphi =\operatorname*{div} (-\Delta)^{-1} \Phi$.  Then $\varphi \in C^1(\mathbb{R}^n)$ and
\begin{align*}
|\varphi(x)| \lesssim \frac{1}{(1+|x|)^{n-1}} \quad \text{and} \quad
|\nabla \varphi(x)| \lesssim \frac{1}{(1+|x|)^{n}} \quad \text{for $x\in \rn$,}
\end{align*}
whence $\varphi \in L^\infty(\mathbb{R}^n)\cap L^{n'p',n'q'}(\mathbb{R}^n)$ and $\nabla \varphi \in L^\infty(\mathbb{R}^n, \rn)\cap L^{p',q'}(\mathbb{R}^n, \rn)$.  Therefore, as shown above, equation \eqref{div_free} holds with this choice of $\varphi$, namely
\begin{align*}
\int_{\mathbb{R}^n} F \cdot \nabla \operatorname*{div} (-\Delta)^{-1} \Phi \;dx =0.
\end{align*}
%However, as  $\Phi \in C^\infty_c(\mathbb{R}^n;\mathbb{R}^n)$, 
%$\nabla \operatorname*{div} (-\Delta)^{-1} \Phi = T\Phi$ and so
Hence, by equation \eqref{aug17}, 
\begin{align*}
\int_{\mathbb{R}^n} F \cdot \Phi \;dx =\int_{\mathbb{R}^n} F \cdot (\Phi+\mathcal H\Phi) \;dx
=\int_{\mathbb{R}^n} F \cdot P\Phi \;dx =\int_{\mathbb{R}^n} PF \cdot \Phi \;dx,
\end{align*}
where the last equality follows via an integration by parts and Fubini's theorem. This proves equation \eqref{aug20}.
\\ Part (ii). We have to show that, 
if $F,PF \in L^1(\rn, \rn)$ or $F,PF \in L^{p,q}(\rn, \rn)$, then
\begin{align}\label{aug23}
\int_{\mathbb{R}^n} PF \cdot \nabla \varphi \;dx =0
\end{align}
for all $\varphi \in C^\infty_c(\mathbb{R}^n)$.  One has that
\begin{align}\label{aug22}
\int_{\mathbb{R}^n} PF \cdot \nabla \varphi \;dx &= \int_{\mathbb{R}^n} F \cdot \nabla \varphi \;dx+ \int_{\mathbb{R}^n} \mathcal H F \cdot \nabla \varphi \;dx.
\end{align}
The assumptions on $F,PF$ imply $ \mathcal H F \in L^1(\rn, \rn)$ or $ \mathcal H F\in L^{p,q}(\rn, \rn)$. Either of these  integrability properties  suffices to ensure that 
\begin{align*}
    \int_{\mathbb{R}^n} \mathcal H F \cdot \nabla \varphi \;dx &= \int_{\mathbb{R}^n}   F \cdot \mathcal H\, \nabla \varphi \;dx.
    %\\
    %&=-\int_{\mathbb{R}^n}   F \cdot \nabla \varphi \;dx,
\end{align*}
Combining the latter equality with 
 equations \eqref{aug18} and \eqref{aug22} yields \eqref{aug23}.
 \end{proof}

\begin{proof}[Proof of Theorem \ref{lemma2}] All functions appearing throughout this proof map $\rn$ into $\rn$. Since there will be no ambiguity, we drop the notation of the domain and the target in the function spaces. The  constants in the relations $\lq\lq \lesssim"$ and $\lq\lq \approx"$ only depend on $n, p, q$.
\\
    Recall that, according to the definition \eqref{Kfunct}, 
    \begin{align*}
        K(F,t,L^1,L^{p,q}) &= \inf \{ \|F_1\|_{L^1}+ t\|F_{p,q}\|_{L^{p,q} } :F=F_1+F_{p,q}\}\quad \text{for $t>0$}
        \end{align*}
        for $F\in L^1 +L^{p,q}$,
        and 
           \begin{align*}
         K(F,t,L^1_{\operatorname*{div}} ,L^{p,q}_{\operatorname*{div}} )  = \inf \{ \|F_1\|_{L^1 }+ t\|F_{p,q}\|_{L^{p,q} } : F=F_1+F_{p,q}, \operatorname*{div} F_1=\operatorname*{div}F_{p,q} =0 \}\quad \text{for $t>0$}
    \end{align*}
    for $F\in L^1_{\operatorname*{div}} +L^{p,q}_{\operatorname*{div}}$. 
    \\ Fix $t>0$ and $F \in L^1   +L^{p,q}$ such that $\operatorname*{div}F=0$. Thanks to a larger class of admissible decompositions in the computation of the $K$-functional on its left-hand side, 
 the inequality
    \begin{align*}
       K(F,t,L^1 ,L^{p,q} ) \leq K(F,t,L^1_{\operatorname*{div}},L^{p,q}_{\operatorname*{div}})
    \end{align*}
    is trivial.  To establish the first equivalence in \eqref{sep1}, it therefore remains to show that, up to a multiplicative constant, the reverse inequality also holds, namely:
    \begin{align}\label{sep5}
    K(F,t,L^1_{\operatorname*{div}},L^{p,q}_{\operatorname*{div}})  \lesssim K(F,t,L^1,L^{p,q}).
    \end{align}
By scaling, we may assume
\begin{align}\label{sep7}
    K(F,t,L^1,L^{p,q}) =1,
    \end{align}
and then we must show that
\begin{align}\label{sep6}
   K(F,t,L^1_{\operatorname*{div}},L^{p,q}_{\operatorname*{div}}) \lesssim 1.
    \end{align}

In order to prove the inequality \eqref{sep6} under \eqref{sep7}, 
consider any decomposition  $F=F_1+F_{p,q}$  of $F$   such that
\begin{align*}
   \|F_1\|_{L^1}+ t\|F_{p,q}\|_{L^{p,q}} \leq 2.
\end{align*}
The Calder\'on-Zygmund decomposition \cite[Theorem 5.3.1 on p.~355]{grafakos} of $F_1$, with $\lambda = t^{-p'}$, yields $$F_1=H+K$$ for some functions $H, K \in L^1$ such that:
\begin{align*}
|H| &\leq \lambda ,\\
\|H\|_{L^1} &\leq \|F_1\|_{L^1} \leq 2,
\end{align*}
and $$K=\sum_i K_i$$ for some functions $K_i\in L^1$ satisfying, for suitable balls $B_i\subset\R^n$,
\begin{align*}
\operatorname*{supp} K_i &\subset B_i,
\\ \int_{B_i} K_i \,dx&= 0,\\
\sum_i |B_i| &\lesssim \|F_1\|_{L^1}
\lambda^{-1} \leq 2\lambda^{-1},\\
\sum_i \|K_i\|_{L^1} &\lesssim \|F_1\|_{L^1} \leq 2.
\end{align*}
By Lemma  \ref{lemma1}, Part (i), we have that $F=PF$. Therefore,
\begin{align}\label{aug1}
F= PF = P(H+F_{p,q})+PK.
\end{align}
If we 
 show that $PK \in L^1, P
(H+F_{p,q})\in L^{p,q}$, and 
\begin{align}\label{final_claim}
\|PK\|_{L^1} + t\|P(H+F_{p,q})\|_{L^{p,q}} \lesssim 1,
\end{align}
then we can conclude that 
\eqref{aug1} is an admissible decomposition for the $K$-functional for the couple $(L^1_{\operatorname*{div}} ,L^{p,q}_{\operatorname*{div}})$, since, by  Lemma  \ref{lemma1}, Part (ii),
$${\operatorname*{div}}PK=0 \quad \text{and} \quad {\operatorname*{div}} P
(H+F_{p,q})=0.$$
Hence \eqref{sep6} will follow via \eqref{final_claim}.
\iffalse

Indeed,  this gives a decomposition of $F$ in divergence free spaces, since, once we have shown the inequality \eqref{final_claim}, Lemma  \ref{lemma2}, Part (ii) implies that
$${\operatorname*{div}}PK=0 \quad \text{and} \quad {\operatorname*{div}} P
(H+F_{p,q})=0.$$

$PK,P(H+F_{p,q})$ are divergence free in the sense of distributions.   Moreover, Part (i) of the same lemma tells us that
$F=PF$. Therefore,
\begin{align}\label{aug1}
F= PF = P(H+F_{p,q})+PK
\end{align}
is an admissible decomposition of $F$ in the computation of $ K(F,t,L^1_{\operatorname*{div}},L^{p,q}_{\operatorname*{div}})$. 
\fi
\\ To complete the proof, it thus only remains
to prove the bound \eqref{final_claim}.
%
%To this end, we observe that the invariance of $F$ under $P$ implies
%\begin{align*}
%F= F_1+F_{p,q} = (H+F_{p,q})+K = P(H+F_{p,q})+P(K).
%\end{align*}
Concerning the second addend on the left-hand side of  \eqref{final_claim}, by the boundedness of $\mathcal H$, and hence of $P$, on 
 $L^{p,q}$
%
 %bounds for the Calder\'on-Zygmund operator $P$ \cite[Theorem 1.4.19 on p.~61 and Theorem 5.2.7 on p.~339]{grafakos} 
 and the inequality \eqref{aug3},
one has
\begin{align}\label{aug2}
\|P(H+F_{p,q})\|_{L^{p,q}} &  \lesssim \|H\|_{L^{p,q}}+  \|F_{p,q}\|_{L^{p,q}}\\ \nonumber
&\lesssim \|H\|_{L^1}^{1/p}\|H\|_{L^\infty}^{1/p'}+\|F_{p,q}\|_{L^{p,q}}\\ \nonumber
&\lesssim \lambda^{1/p'} +t^{-1}\\ \nonumber
&\approx t^{-1}.
\end{align}
%This completes the proof of the desired bound for the second addend.
Turning our attention to the bound for the first addend on the right-hand side of \eqref{final_claim}, define $\Omega = \cup_i B_i^*$, where $B_i^{*}$ is the ball with the same center as $B_i$ with twice the radius.  Then,
\begin{align*}
\|PK\|_{L^1} = \|PK \chi_\Omega\|_{L^1} + \|PK \chi_{\Omega^c}\|_{L^1}.
\end{align*}
Inasmuch as $\mathrm{supp} \,K_i\subset B_i$ and
 the kernel of the operator $\mathcal H$  satisfies H\"ormander's condition \eqref{hormander}, by \cite[Inequality (2.13)]{Bourgain} 
$$ \sum_i \|PK_i \chi_{(B_i^*)^c}\|_{L^1}  =\sum_i \|\mathcal H K_i \chi_{(B_i^*)^c}\|_{L^1}  \lesssim \sum_i \|K_i\|_{L^1}.$$
Hence,
\iffalse
 
\footnote{For multipliers which satisfy Mihlin's condition this is argued after (6.2.20) on p.~448 of \cite{grafakos}.}, and so as in the argument of Bourgain's Lemma 2.4 we have
\todo[inline]{A: I would mention and explicitly recall these properties when we introduce $T$. I would also repeat  Bourgain's argument}
\fi
\begin{align}\label{aug7}
\|PK \chi_{\Omega^c}\|_{L^1} &\leq \sum_i \|PK_i \chi_{(B_i^*)^c}\|_{L^1}  \lesssim \sum_i \|K_i\|_{L^1} \lesssim \|F\|_{L^1}.
\end{align}
Since $F=H+K+F_{p,q}$, from equation \eqref{aug1} we deduce that
\begin{align*}
PK = K+F_{p,q}+H - P(F_{p,q}+H).
\end{align*}
Therefore,   the boundedness of $P$ on $L^{p,q}(\rn)$, the  H\"older type  inequality in the Lorentz spaces \eqref{holderlor}, and the fact that $\|\chi_{\Omega}\|_{L^{p',q'}}$ is independent of $q$ yield
\begin{align*}
\|PK\chi_\Omega\|_{L^1} &\leq \|K\|_{L^1} + |\Omega|^{1/p'} \|F_{p,q}+H\|_{L^{p,q}}\\
&\lesssim \|F_1\|_{L^1} + |\Omega|^{1/p'} t^{-1} \lesssim 1 + \lambda^{-1/p'}t^{-1}  \approx 1
\end{align*}
by our choice of $\lambda$.  This completes the proof of the bound \eqref{final_claim} and thus also the proof of the first equivalence in \eqref{sep1}.
\\ The second equivalence holds thanks to Holmsted's formulas \cite[Theorem 4.1]{Hol}.

\begin{proof}[Proof of Theorem \ref{K-CZ}, case $k=1$] 
Throughout this proof, the constants in the relations \lq\lq $\approx$"
 and \lq\lq $\lesssim$" depend only on $\frac n\alpha$ and $k$.
Observe that   $F \in  L^1 (\rn, \mathbb{R}^n)+ L^{\frac n\alpha, 1} (\rn, \mathbb{R}^n)$ if and only if 
$$\int_0^{t} F^*(s)\;ds + t^{1-\alpha/n} \int_{t}^\infty  s^{-1+\frac \alpha n} F^*(s)\;ds<\infty \quad \text{for $t>0$.}$$
This is a consequence  Holmsted's formulas -- see the second equivalence in \eqref{sep1}.
On the other hand, an application of Fubini's theorem tells us that 
\begin{align}
    \label{nov130}
     \int_0^{t} F^*(s)\;ds + t^{1-\alpha/n} \int_{t}^\infty  s^{-1+\frac \alpha n} F^*(s)\;ds  =   \frac {n-\alpha}n \int_0^{t} s^{-\alpha/n}   \int_{s}^\infty r^{-1+\alpha/n}F^*(r)\;drds
\end{align}
for $t>0$.
%Hence, we may assume that $F \in  L^1 (\rn, \mathbb{R}^N)+ L^{\frac n\alpha, 1} (\rn, \mathbb{R}^N)$, otherwise the right-hand side of the inequality \eqref{nov100}  is infinite, and the inequality holds trivially.
\\ 
    From \cite[Theorem 1.1]{HS} one has that
 \begin{equation}\label{dec30}
 I_\alpha : L^1_{\operatorname*{div}}(\rn, \R^n) \to L^{\frac n{n-\alpha}, 1}(\rn, \R^n),
 \end{equation}
  with norm depending on $n$ and $\alpha$.  
\\ On the other hand,
\begin{align}\label{cocanc1}
 I_\alpha : L^{\frac n\alpha, 1} (\rn, \rn) \to L^{\infty}(\rn, \rn)
\end{align}
Let $F \in  L^1(\rn, \rn)+ L^{\frac n\alpha, 1}(\rn, \rn)$  be such that ${\rm div}F=0$ row-wise.   By Theorem \ref{lemma2}, such a function $F$ admits a decomposition $F=F_1 + F_{n/\alpha,1}$, with $F_1\in L^{1}_{\rm div}(\rn, \rn)$ and $F_{n/\alpha,1}\in  L^{\frac n\alpha, 1}_{\rm div}(\rn, \rn)$ fulfilling the estimate:
\begin{align}
    \label{nov134}
    \|F_1\|_{L^{1}(\rn, \rn)}+ t \|F_{n/\alpha,1}\|_{L^{\frac n\alpha, 1}(\rn, \rn)} \lesssim  \int_0^{t} F^*(s)\;ds + t^{1-\alpha/n} \int_{t}^\infty  s^{-1+\frac \alpha n} F^*(s)\;ds
\end{align}
for $t>0$.
%\\ Equations \eqref{nov110} and \eqref{nov111} ensure that
%$\sum_{|\beta|=k} I_\alpha T^\beta F^\beta_1 \in L^{\frac{n}{n-\alpha}, 1}(\rn,\R^l)$ and $\sum_{|\beta|=k}  I_\alpha T^\beta F^\beta_{n/\alpha,1} \in L^\infty(\rn,\R^l)$.
Therefore,
\begin{align}
   \nonumber
    K\big( I_\alpha F, t;   L^{\frac{n}{n-\alpha}, 1}(\rn,\R^n),  L^\infty(\rn, \R^n)\big)   \lesssim \|F_1\|_{L^{1}(\rn, \rn)}+ t \|F_{n/\alpha}\|_{L^{\frac n\alpha, 1}(\rn, \rn)}.
    \label{nov132}
\end{align}
 \iffalse
\todo[inline]{A: \eqref{dec32prime} 
it is not straightforward. We have to use the special decomposition from \eqref{aug1} for a function such that $div F=0$, where $F=F_1+F_pq$, with $div F_1=0$ and $div F_{pq} =0$.}
 
The property \eqref{dec37} of the $K$-functional gives
\begin{equation}\label{dec32prime}
K\big({\color{red}I_\alpha T QF}, t;   L^{\frac{n}{n-\alpha}, 1}(\rn, \om),  L^\infty(\rn, \om)\big)
\leq c K\big(F, t/c;   L^1_{\operatorname*{div_k}}(\rn, \om),   L^{\frac n\alpha, 1}_{\operatorname*{div_k}}(\rn, \om)\big),
\end{equation}
for some positive constant  $c=c(n,\alpha, k)$ and for every $F \in  L^1_{\operatorname*{div_k}}(\rn, \mathbb{R}^n)+ L^{\frac n\alpha, 1}_{\operatorname*{div_k}}(\rn, \mathbb{R}^n)$.  
\fi
Thanks to \cite[Corollary 2.3, Chapter 5]{BennettSharpley}, 
\begin{equation}\label{dec33}
K\big( G, t;   L^{\frac{n}{n-\alpha}, 1}(\rn,\R^n),  L^\infty(\rn, \R^n)\big) \approx \int_0^{t^{\frac n{n-\alpha}}} s^{-\frac{\alpha}{n}}G^*(s) \, ds \quad\text{for $t>0$,}
\end{equation} 
for $G\in L^{\frac n{n-\alpha}, 1}(\rn, \rn)+  L^\infty(\rn, \rn)$. 
\\ Combining equations \eqref{nov134}--\eqref{dec33}   with \eqref{nov130} yields the inequality \eqref{nov100}.
\end{proof}

\iffalse

the well-known formula for the $K$-functional for the spaces $L^1(\rn, \rM)$ and $L^\infty(\rn, \rM)$ and of the
reiteration theorem for $K$-functionals  \cite[Corollary 3.6.2]{BerghLofstrom}
\fi
\end{proof}

\section{Proofs of the main results}\label{proof}

The proof of Theorem \ref{sobolev} is reduced to the case of $k$-th order divergence free vector fields thanks to the following lemma.
%It enables one to transfer the information contained in any Riesz potential estimate for $k$-th order divergence free vector fields to a parallel estimate for every $k$-th order co-canceling operator.

\begin{lemma}\label{reductionk}
Let $n, m, k \in \N$, with $m,n \geq 2$, $ k \in \mathbb{N}$, $\alpha \in (0,n)$,     and let $N$ be defined by \eqref{N}. Assume that $\|\cdot\|_{X(0,\infty)}$ and $\|\cdot\|_{Y(0,\infty)}$ are rearrangement-invariant function norms and   let $\mathcal L$ be any linear homogeneous $k$-th order  co-canceling differential operator.
Suppose that there exists a constant $c_1$ such that
\begin{align}\label{div_k_free_estimate}
    \|I_\alpha F\|_{Y(\mathbb R^n, \om)} \leq c_1 \|F\|_{X(\rn, \om)}
\end{align}
for all $F \in X_{\operatorname*{div}_k}(\rn, \om)$.  Then,
\begin{align}\label{L_k__free_estimate}
    \|I_\alpha F\|_{Y(\mathbb R^n, \mathbb R^m)} \leq c_2 \|F\|_{X(\rn, \mathbb R^m)}
\end{align}
for some constant $c_2=c_2(c_1, \mathcal L)$ and all $F \in X_{\mathcal{L}}(\rn, \mathbb R^m)$.
\end{lemma}

\begin{proof}
 Let $\mathcal L(D)$ be as in Definition \ref{co-canc}. Fix any function $F\in X_{\mathcal L}(\rn, \rM)$.  As in \cite[Lemma 2.5]{VS3}, one has that
 $$
 \mathcal L(D)F=\sum_{\beta \in \mathbb N^n, |\beta|=k} L_\beta\partial^\beta F=\sum_{\beta \in \mathbb N^n, |\beta|=k} \partial^\beta (L_\beta F)   =0
 %\vec{0},
 $$ 
 for  suitable linear maps $L_\beta \in\operatorname{Lin}(\R^m, \R^l)\simeq\R^{l\times m}$ independent of $F$, and suitable  $l \in \mathbb N$.
 %, where $\vec{0}$ is the zero vector in $\mathbb{R}^l$.   
 In analogy with the previous section, write $LF = [L_\beta F]_{|\beta|=k} \in \operatorname{Lin} (\mathbb{R}^n, \R^{N\times l})$ for the collection of $l$ maps with values in $\R^{N}$.  In particular, one can regard $[L_\beta F]$ as $l$ rows  $\{(L_\beta F)_i\}_{i=1}^l$ such that $\operatorname*{div}_k (L_\beta F)_i=0$ for each $i=1,\ldots,l$. 
\\
Owing to \cite[Lemma 2.5]{VS3} there exist a family of maps $K_\beta \in \operatorname{Lin}(\mathbb{R}^l, \mathbb{R}^m)$ such that
\begin{align*}
F = \sum_{|\beta|=k} K_\beta L_\beta F.
\end{align*}
Hence,
\begin{align*}
|I_\alpha F| = \left|\sum_{|\beta|=k} K_\beta I_\alpha L_\beta F\right| \lesssim \sum_{|\beta|=k} \left|I_\alpha L_\beta F\right| \lesssim {\sum_{|\beta|=k}} \sum_{i=1}^l \left|I_\alpha (L_\beta F)_i\right|.
\end{align*}
Thus, by the property (P2) of the function norm $Y(0,\infty)$, one has
\begin{align*}
    \|I_\alpha F\|_{Y(\mathbb R^n, \mathbb R^m)} \leq c{\sum_{|\beta|=k}}\sum_{i=1}^l  \|I_\alpha (L_\beta F)_i\|_{Y(\rn, \mathbb R^{{N}})}.
\end{align*}
An application of the inequality \eqref{div_k_free_estimate} for each $i=1,\ldots,l$
yields
\begin{align}\label{oct7}
  \|I_\alpha (L_\beta F)_i\|_{Y(\rn, \mathbb R^{{N}})} \leq c'  \|  (L_\beta F)_i\|_{X(\rn, \mathbb R^{N})}
\end{align}
for some constant $c'=c'(c_1, \mathcal L)$.
Since each of the linear maps $L_\beta$ is bounded, the inequalities \eqref{oct7} and the  property (P2) of the rearrangement-invariant function norm $X(0, \infty)$ yield 
  \eqref{L_k__free_estimate} for an arbitrary $k$-th order co-canceling operator $\mathcal L$.
\end{proof}

 The following 
 result from \cite[Proof of Theorem A]{KermanPick} (see also \cite[Proof of Theorem 4.1]{CPS_Frostman} for an alternative simpler proof) is needed
  to combine the information contained in the inequality \eqref{nov100} with the assumption \eqref{bound1}.
 % It enables one 
  %to transfer the information contained in a pointwise  inequality  between Hardy type operators applied to two functions satisfying the inequality \eqref{bound1} for some rearrangement-invariant function norms $\|\cdot\|_{X(0,\infty)}$ and $\|\cdot\|_{Y(0,\infty)}$ into 
% an inequality between the same  norms of the two functions.

 \begin{theoremalph}\label{KermanPick} Let $n\in\mathbb N$ and $\alpha\in(0,n)$. Let $\|\cdot\|_{X(0,\infty)}$ and $\|\cdot\|_{Y(0,\infty)}$ be  rearrangement-invariant function norms such that the inequality \eqref{bound1} holds. Suppose that the functions $f, g \in \mathcal M (0, \infty)$ are such that
\begin{align}\label{july21}
\int_0^t  s^{-\frac{\alpha}{n}} g^*(s)\, ds
\leq c\int_0^{t}  s^{-\frac{\alpha}{n}} \int_{s/c}^{\infty} f^*(r)r^{-1+\alpha/n}\,drds  \quad\text{for $t>0$,}
\end{align}
for some positive constant   $c$.  Then
\begin{align}\label{july22}
\|g\|_{Y(0, \infty)} \leq c' \|f\|_{X(0, \infty)},
\end{align}
for a suitable constant $c'=c'(c, \frac n\alpha)$. 
 \end{theoremalph}

%\subsection{Riesz potential estimates under co-canceling constraints}
We are now in a position to accomplish the proof of Theorem \ref{sobolev}.

 \begin{proof}[Proof of Theorem \ref{sobolev}]  
 To begin with, as observed with regard to the condition \eqref{may8}, such a condition is necessarily fulfilled
 if 
 the inequality \eqref{bound1}
holds for some rearrangement-invariant function norms $\|\cdot\|_{X(0,\infty)}$ and $\|\cdot\|_{Y(0,\infty)}$. Hence, thanks to   the H\"older type inequality \eqref{holder}, if 
 $F\in X(\rn, \rM)$, then
 \begin{align}
     \label{dec24-1}
     %\int_{\rn}|F(x)|(1+|x|)^{-\alpha+n}\, dx & \lesssim 
     \int_0^\infty F^*(s)(1+s)^{-1+\frac \alpha n}\, ds & \leq \|F^*\|_{X(0, \infty)}\|(1+s)^{-1+\frac \alpha n}\|_{X'(0,\infty)}
     \\ \nonumber & = \|F\|_{X(\rn, \rM)}\|(1+s)^{-1+\frac \alpha n}\|_{X'(0,\infty)}<\infty.
 \end{align}
 The finiteness of the leftmost side of the chain \eqref{dec24-1}
implies that $F\in L^1(\rn, \rM)+ L^{\frac n\alpha , 1}(\rn, \rM)$. 
 \\ Next, as a first step, we  consider that case when $$\mathcal L = \operatorname*{div_k}.$$ 
 Namely, we assume that the inequality \eqref{bound1} holds and we shall prove that the inequality \eqref{bound2} is satisfied for all
    $F\in X_{\operatorname*{div_k}}(\rn, \om)$.
    By Theorem \ref{K-CZ}, 
   one has that
\begin{align}\label{dec36}
 \int_0^{t} s^{-\frac{\alpha}{n}}(I_\alpha  F)^*(s) \, ds \leq c  \int_0^{t} s^{-\alpha/n}   \int_{s/c}^\infty r^{-1+\alpha/n}F^*(r)\;drds  \qquad \text{for $t>0$,}
\end{align}
for some positive constant $c=c(n, \alpha, k)$. The inequality \eqref{bound2}
%for  $F\in  L^1_{\operatorname*{div}}(\rn) +  L^{\frac n\alpha, 1}_{\operatorname*{div}}(\rn)$ then 
follows from \eqref{dec36}, via Theorem \ref{KermanPick}.  Thereby,  we have shown that
\begin{align}\label{sep15}
    \|I_\alpha F\|_{Y(\mathbb R^n, \om)} \leq c \|F\|_{X(\rn, \om)}
\end{align}
for some constant $c=c(n, \alpha, k)$ and every $F\in X_{\operatorname*{div_k}}(\rn, \om)$. 
 The inequality \eqref{bound2} for $F\in X_{\mathcal L}(\rn, \rM)$, where $\mathcal{L}$ is any  linear homogeneous $k$-th order  co-canceling operator, is a consequence of \eqref{sep15} and of Lemma \ref{reductionk}.
 \end{proof}

\iffalse 

\begin{proof}[Proof of Theorem \ref{sobolev-opt}] Thanks to \cite[Theorem~4.4]{EMMP}, the inequality \eqref{bound1} holds with $Y(0,\infty)= X_\alpha (0, \infty)$. Hence, the inequality \eqref{bound-opt} follows from an application of Theorem \ref{sobolev}.
\end{proof}

\fi

\begin{proof}[Proof of Corollary \ref{sobolev-domain}] Assume that the inequality \eqref{bound1-domain} holds.  We claim that 
\begin{equation}\label{sep27}
\bigg\|\int_s^\infty r^{-1+\frac \alpha n}f(r)\, dr \bigg\|_{Y^e(0,\infty)} \leq c_1 \|f\|_{X^e(0, \infty)}
\end{equation}
for every $f \in \Mpl(0,\infty)$ with ${\rm supp} f \subset [0, |\Omega|]$, where $X^e(0, \infty)$ and $Y^e(0, \infty)$ denote the extended function norms defined as in \eqref{sep25}. Indeed, the inequality \eqref{sep27} can be verified via the following chain:
\begin{align}
    \label{sep28}
    \bigg\|\int_s^\infty r^{-1+\frac \alpha n}f(r)\, dr \bigg\|_{Y^e(0,\infty)} & = 
    \bigg\|\bigg(\chi_{[0,|\Omega|]}(\cdot)\int_{(\cdot)}^\infty r^{-1+\frac \alpha n}f(r)\, dr \bigg)^*(s)\bigg\|_{Y^e(0,\infty)}
    \\ \nonumber & = \bigg\| \int_{s}^{|\Omega|} r^{-1+\frac \alpha n}f(r)\, dr  \bigg\|_{Y(0,|\Omega|)}
    \\ \nonumber & \leq c_1 \|f\|_{X(0, |\Omega|)} = c_1\|f^*\|_{X(0, |\Omega|)} = c_1\|f\|_{X^e(0, \infty)}.
\end{align}
Now, assume that $F\in X_{\mathcal L} (\rn, \rM)$ is such that $F=0$ a.e. in $\rn \setminus \Omega$. An application of Theorem \ref{sobolev} tells us that 
\begin{equation}\label{sep32}
\|I_\alpha F\|_{Y^e(\mathbb R^n, \mathbb R^m)} \leq c_2 \|F\|_{X^e(\rn, \mathbb R^m)}
\end{equation}
for some constant $c_2=c_2(c_1, n, \alpha)$.
On the other hand,
\begin{align}
    \label{sep30}
    \|F\|_{X^e(\rn, \rM)}= \|F^*\|_{X^e(0, \infty)}= \|F^*\|_{X(0, |\Omega|)} = \|F\|_{X(\Omega, \rM)},
\end{align}
and
\begin{align}
    \label{sep31}
    \|I_\alpha F\|_{Y^e(\rn, \rM)}=\|(I_\alpha F)^*\|_{Y^e(0, \infty)}
    = \|(I_\alpha F)^*\|_{Y(0,|\Omega|)}\geq 
    \|(\chi_\Omega I_\alpha F)^*\|_{Y(0,|\Omega|)}= \|I_\alpha F\|_{Y(\Omega, \rM)}.
\end{align}
Combining equations \eqref{sep32}--\eqref{sep31} yields \eqref{bound2-domain}.
\end{proof}

\bigskip

\bigskip{}{}

 \par\noindent {\bf Data availability statement.} Data sharing not applicable to this article as no datasets were generated or analysed during the current study.

\section*{Compliance with Ethical Standards}\label{conflicts}

\smallskip
\par\noindent
{\bf Funding}. This research was partly funded by:
\\
(i) Grant BR 4302/3-1 (525608987) by the German Research Foundation (DFG) within the framework of the priority research program SPP 2410 (D. Breit);\\
(ii) Grant BR 4302/5-1 (543675748) by the German Research Foundation (DFG) (D. Breit);
\\ (iii) GNAMPA   of the Italian INdAM - National Institute of High Mathematics (grant number not available)  (A. Cianchi);
\\ (iv) Research Project   of the Italian Ministry of Education, University and
Research (MIUR) Prin 2017 ``Direct and inverse problems for partial differential equations: theoretical aspects and applications'',
grant number 201758MTR2 (A. Cianchi);
\\ (v) Research Project   of the Italian Ministry of Education, University and
Research (MIUR) Prin 2022 ``Partial differential equations and related geometric-functional inequalities'',
grant number 20229M52AS, cofunded by PNRR (A. Cianchi);
\\  (vi) National Science and Technology Council of Taiwan research grant numbers 110-2115-M-003-020-MY3/113-2115-M-003-017-MY3 (D. Spector);
\\ (vii) Taiwan Ministry of Education under the Yushan Fellow Program (D. Spector).

\bigskip
\par\noindent
{\bf Conflict of Interest}. The authors declare that they have no conflict of interest.


\begin{bibdiv}

\begin{biblist}




\bib{ACPS}{article}{
  author={Alberico, Angela},
  author={Cianchi, Andrea},
  author={Pick, Lubos},
  author={Slav\'ikov\'a, Lenka},
  title={Fractional Orlicz-Sobolev embeddings},
  journal={J. Math. Pures Appl.},
   volume={149},
  date={2021},
  number={7},
  pages={539--543},
%issn={1631-073X},
  %doi={10.1016/j.crma.2003.12.031},
}



\bib{ACPS_NA}{article}{
   author={Alberico, Angela},
   author={Cianchi, Andrea},
   author={Pick, Lubo\v{s}},
   author={Slav\'{\i}kov\'{a}, Lenka},
   title={Boundedness of functions in fractional Orlicz-Sobolev spaces},
   journal={Nonlinear Anal.},
   volume={230},
   date={2023},
   pages={Paper No. 113231, 26},
   issn={0362-546X},
   review={\MR{4551936}},
   doi={10.1016/j.na.2023.113231},
}

 


\bib{BennettSharpley}{book}{
   author={Bennett, Colin},
   author={Sharpley, Robert},
   title={Interpolation of operators},
   series={Pure and Applied Mathematics},
   volume={129},
   publisher={Academic Press, Inc., Boston, MA},
   date={1988},
   pages={xiv+469},
   isbn={0-12-088730-4},
   review={\MR{928802}},
}

% \bib{BerghLofstrom}{book}{
%    author={Bergh, J\"{o}ran},
%    author={L\"{o}fstr\"{o}m, J\"{o}rgen},
%    title={Interpolation spaces. An introduction},
%    series={Grundlehren der Mathematischen Wissenschaften},
%    volume={No. 223},
%    publisher={Springer-Verlag, Berlin-New York},
%    date={1976},
%    pages={x+207},
%    review={\MR{0482275}},
% }



\bib{Bourgain}{article}{
  author={Bourgain, Jean},
  title={Some consequences of Pisier's approach to interpolation},
  journal={Israel Journal of Mathematics},
   volume={77},
  date={1992},
  pages={165--185},
%issn={1631-073X},
  %doi={10.1016/j.crma.2003.12.031},
}

\bib{BourgainBrezis2004}{article}{
  author={Bourgain, Jean},
  author={Brezis, Ha{\"{\i}}m},
  title={New estimates for the Laplacian, the div-curl, and related Hodge
  systems},
  journal={C. R. Math. Acad. Sci. Paris},
   volume={338},
  date={2004},
  number={7},
  pages={539--543},
%issn={1631-073X},
  %doi={10.1016/j.crma.2003.12.031},
}


\bib{BourgainBrezis2007}{article}{
   author={Bourgain, Jean},
   author={Brezis, Ha{\"{\i}}m},
   title={New estimates for elliptic equations and Hodge type systems},
   journal={J. Eur. Math. Soc. (JEMS)},
   volume={9},
  date={2007},
  number={2},
   pages={277--315},
  }
  
  \bib{BourgainBrezisMironescu}{article}{
   author={Bourgain, Jean},
   author={Brezis, Haim},
   author={Mironescu, Petru},
   title={$H^{1/2}$ maps with values into the circle: minimal
   connections, lifting, and the Ginzburg-Landau equation},
   journal={Publ. Math. Inst. Hautes \'Etudes Sci.},
   number={99},
   date={2004},
   pages={1--115},
  % issn={0073-8301},
}



\bib{Breit-Cianchi}{article}{
   author={Breit, Dominic},
   author={Cianchi, Andrea},
   title={Symmetric gradient Sobolev spaces endowed with
   rearrangement-invariant norms},
   journal={Adv. Math.},
   volume={391},
   date={2021},
   pages={Paper No. 107954, 101},
   issn={0001-8708},
   review={\MR{4303731}},
   doi={10.1016/j.aim.2021.107954},
}



\bib{BCS_canceling}{article}{AUTHOR = {Breit, Dominic},
    AUTHOR= {Cianchi, Andrea}, AUTHOR= {Spector, Daniel}
   title={Sobolev inequalities for canceling operators},
   journal={preprint at arXiv:2501.07874v2},
   volume={},
   date={},
   number={},
   pages={},
   issn={},
   review={},
}

% \bib{BrDF}{article}{
%     AUTHOR = {Breit, Dominic},
%     AUTHOR= {Diening, Lars}, 
%     AUTHOR= {Fuchs, Martin},
%      TITLE = {Solenoidal {L}ipschitz truncation and applications in fluid
%               mechanics},
%    JOURNAL = {J. Differential Equations},
%   FJOURNAL = {Journal of Differential Equations},
%     VOLUME = {253},
%       YEAR = {2012},
%     NUMBER = {6},
%      PAGES = {1910--1942},
%       ISSN = {0022-0396,1090-2732},
%    MRCLASS = {35Q35 (35A01 35D30 35J60 76A05)},
%   MRNUMBER = {2943947},
% MRREVIEWER = {Ilia\ S.\ Mogilevski\u{\i}},
%        DOI = {10.1016/j.jde.2012.05.010},
%        URL = {https://doi.org/10.1016/j.jde.2012.05.010},
% }




% \bib{CZ}{article}{
%    author={Calderon, A. P.},
%    author={Zygmund, A.},
%    title={On the existence of certain singular integrals},
%    journal={Acta Math.},
%    volume={88},
%    date={1952},
%    pages={85--139},
%    issn={0001-5962},
%    review={\MR{52553}},
%    doi={10.1007/BF02392130},
% }
		


\bib{cavaliere-new}{article}{AUTHOR = {Cavaliere, Paola},
    AUTHOR= {Drazny, Ladislav},
   title={Sobolev embeddings for Lorentz-Zygmund spaces },
   journal={preprint},
   volume={},
   date={},
   number={},
   pages={},
   issn={},
   review={},
}


 


\bib{cianchi_IUMJ}{article}{
   author={Cianchi, Andrea},
   title={A sharp embedding theorem for Orlicz-Sobolev spaces},
   journal={Indiana Univ. Math. J.},
   volume={45},
   date={1996},
   number={1},
   pages={39--65},
   issn={0022-2518},
   review={\MR{1406683}},
   doi={10.1512/iumj.1996.45.1958},
}
	

\bib{cianchi_CPDE}{article}{
   author={Cianchi, Andrea},
   title={Boundedness of solutions to variational problems under general
   growth conditions},
   journal={Comm. Partial Differential Equations},
   volume={22},
   date={1997},
   number={9-10},
   pages={1629--1646},
   issn={0360-5302},
   review={\MR{1469584}},
   doi={10.1080/03605309708821313},
}
	

\bib{Cianchi-JLMS}{article}{
   author={Cianchi, Andrea},
   title={Strong and weak type inequalities for some classical operators in
   Orlicz spaces},
   journal={J. London Math. Soc. (2)},
   volume={60},
   date={1999},
   number={1},
   pages={187--202},
   issn={0024-6107},
   review={\MR{1721824}},
   doi={10.1112/S0024610799007711},
}


% \bib{cianchi_PJM}{article}{
%    author={Cianchi, Andrea},
%    title={A fully anisotropic Sobolev inequality},
%    journal={Pacific J. Math.},
%    volume={196},
%    date={2000},
%    number={2},
%    pages={283--295},
%    issn={0030-8730},
%    review={\MR{1800578}},
%    doi={10.2140/pjm.2000.196.283},
% }
	

\bib{cianchi_ibero}{article}{
   author={Cianchi, Andrea},
   title={Optimal Orlicz-Sobolev embeddings},
   journal={Rev. Mat. Iberoamericana},
   volume={20},
   date={2004},
   number={2},
   pages={427--474},
   issn={0213-2230},
   review={\MR{2073127}},
   doi={10.4171/RMI/396},
}

\bib{cianchi_forum}{article}{
   author={Cianchi, Andrea},
   title={Higher-order Sobolev and Poincar\'{e} inequalities in Orlicz spaces},
   journal={Forum Math.},
   volume={18},
   date={2006},
   number={5},
   pages={745--767},
   issn={0933-7741},
   review={\MR{2265898}},
   doi={10.1515/FORUM.2006.037},
}

% \bib{cianchi-pick-slavikova}{article}{
%    author={Cianchi, Andrea},
%    author={Pick, Lubo\v{s}},
%    author={Slav\'{\i}kov\'{a}, Lenka},
%    title={Higher-order Sobolev embeddings and isoperimetric inequalities},
%    journal={Adv. Math.},
%    volume={273},
%    date={2015},
%    pages={568--650},
%    issn={0001-8708},
%    review={\MR{3311772}},
%    doi={10.1016/j.aim.2014.12.027},
% }

\bib{CPS_Frostman}{article}{
   author={Cianchi, Andrea},
   author={Pick, Lubo\v{s}},
   author={Slav\'{\i}kov\'{a}, Lenka},
   title={Sobolev embeddings, rearrangement-invariant spaces and Frostman
   measures},
   language={English, with English and French summaries},
   journal={Ann. Inst. H. Poincar\'{e} C Anal. Non Lin\'{e}aire},
   volume={37},
   date={2020},
   number={1},
   pages={105--144},
   issn={0294-1449},
   review={\MR{4049918}},
   doi={10.1016/j.anihpc.2019.06.004},
}

\bib{DG}{article}{
    AUTHOR = {Diening, Lars},
    AUTHOR = {Gmeineder, Franz},
     TITLE = {Continuity points via {R}iesz potentials for
              {$\mathbb{C}$}-elliptic operators},
   JOURNAL = {Q. J. Math.},
%  FJOURNAL = {The Quarterly Journal of Mathematics},
    VOLUME = {71},
      YEAR = {2020},
    NUMBER = {4},
     PAGES = {1201--1218},
%      ISSN = {0033-5606,1464-3847},
 %  MRCLASS = {26B30 (31C45)},
 % MRNUMBER = {4186516},
%MRREVIEWER = {Juha\ K.\ Kinnunen},
       DOI = {10.1093/qmathj/haaa027},
       URL = {https://doi.org/10.1093/qmathj/haaa027},
}

\bib{DG2}{article}{
    AUTHOR = {Diening, Lars},
    AUTHOR = {Gmeineder, Franz},
     TITLE = {Sharp Trace and Korn Inequalities for Differential Operators},
   JOURNAL = {Potent. Anal.},
%  FJOURNAL = {Potential Analysis},
    VOLUME = {},
      YEAR = {2024},
    NUMBER = {},
     PAGES = {},
%      ISSN = {},
%   MRCLASS = {},
%  MRNUMBER = {},
%MRREVIEWER = {},
       DOI = {10.1007/s11118-024-10165-1},
       URL = {https://doi.org/10.1007/s11118-024-10165-1},
}


% \bib{DRW}{article}{
%     AUTHOR = {Diening, Lars},
% AUTHOR= {R\r{u}\v{z}i\v{c}ka, Michael},
%  AUTHOR= {Wolf, J\"{o}rg},
%      TITLE = {Existence of weak solutions for unsteady motions of
%               generalized {N}ewtonian fluids},
%    JOURNAL = {Ann. Sc. Norm. Super. Pisa Cl. Sci. (5)},
%  % FJOURNAL = {Annali della Scuola Normale Superiore di Pisa. Classe di
%  %             Scienze. Serie V},
%     VOLUME = {9},
%       YEAR = {2010},
%     NUMBER = {1},
%      PAGES = {1--46},
%       ISSN = {0391-173X,2036-2145},
%  %  MRCLASS = {76D03 (34A34 35D30 35Q35 76A05)},
% %  MRNUMBER = {2668872},
% %MRREVIEWER = {Miroslav\ Bul\'{\i}\v{c}ek},
% }



\bib{EGP}{article}{
   author={Edmunds, David E.},
   author={Gurka, Petr},
   author={Pick, Lubo\v{s}},
   title={Compactness of Hardy-type integral operators in weighted Banach
   function spaces},
   journal={Studia Math.},
   volume={109},
   date={1994},
   number={1},
   pages={73--90},
   issn={0039-3223},
   review={\MR{1267713}},
}

% \bib{EKP}{article}{
%    author={Edmunds, D. E.},
%    author={Kerman, R.},
%    author={Pick, L.},
%    title={Optimal Sobolev imbeddings involving rearrangement-invariant
%    quasinorms},
%    journal={J. Funct. Anal.},
%    volume={170},
%    date={2000},
%    number={2},
%    pages={307--355},
%    issn={0022-1236},
%    review={\MR{1740655}},
%    doi={10.1006/jfan.1999.3508},
% }




\bib{EMMP}{article}{
   author={Edmunds, David E.},
   author={Mihula, Zden\v{e}k},
   author={Musil, V\'{\i}t},
   author={Pick, Lubo\v{s}},
   title={Boundedness of classical operators on rearrangement-invariant
   spaces},
   journal={J. Funct. Anal.},
   volume={278},
   date={2020},
   number={4},
   pages={108341, 56},
   issn={0022-1236},
   review={\MR{4044737}},
   doi={10.1016/j.jfa.2019.108341},
}



\bib{federer}{article}{
   author={Federer, Herbert},
   author={Fleming, Wendell H.},
   title={Normal and integral currents},
   journal={Ann. of Math. (2)},
   volume={72},
   date={1960},
   pages={458--520},
   issn={0003-486X},
   review={\MR{123260}},
   doi={10.2307/1970227},
}
	
\bib{gagliardo}{article}{
   author={Gagliardo, Emilio},
   title={Propriet\`a di alcune classi di funzioni in pi\`u variabili},
   language={Italian},
   journal={Ricerche Mat.},
   volume={7},
   date={1958},
   pages={102--137},
   issn={0035-5038},
   review={\MR{102740}},
}


\bib{GRV}{article}{
    AUTHOR = {Gmeineder, Franz},
    AUTHOR = {Rai\c t\u a, Bogdan},
    AUTHOR = {Van Schaftingen, Jean},
     TITLE = {On limiting trace inequalities for vectorial differential
              operators},
   JOURNAL = {Indiana Univ. Math. J.},
%  FJOURNAL = {Indiana University Mathematics Journal},
    VOLUME = {70},
      YEAR = {2021},
    NUMBER = {5},
     PAGES = {2133--2176},
      ISSN = {0022-2518,1943-5258},
%   MRCLASS = {35J05 (46E35)},
%  MRNUMBER = {4340491},
       DOI = {10.1512/iumj.2021.70.8682},
       URL = {https://doi.org/10.1512/iumj.2021.70.8682},
}
	

% \bib{FV}{article}{
%    author={Ferrari, Fausto and Verbitsky, Igor E.},
%    title={Radial fractional Laplace operators and Hessian inequalites},
%    journal={},
%    volume={},
%    date={},
%    number={},
%    pages={},
%    issn={},
%    review={},
%    doi={arXiv:1203.3149},
% }

% \bib{Gagliardo}{article}{
%    author={Gagliardo, Emilio},
%    title={Propriet\`a di alcune classi di funzioni in pi\`u variabili},
% %    language={Italian},
%    journal={Ricerche Mat.},
%    volume={7},
%    date={1958},
%    pages={102--137},
%    issn={0035-5038},
%    %review={\MR{0102740 (21 \#1526)}},
% }

\bib{grafakos}{book}{
   author={Grafakos, Loukas},
   title={Classical Fourier analysis},
   series={Graduate Texts in Mathematics},
   volume={249},
   edition={3},
   publisher={Springer, New York},
   date={2014},
   pages={xviii+638},
   isbn={978-1-4939-1193-6},
   isbn={978-1-4939-1194-3},
   review={\MR{3243734}},
   doi={10.1007/978-1-4939-1194-3},
}

% \bib{Greco-Moscariello}{article}{
%    author={Greco, L.},
%    author={Moscariello, G.},
%    title={An embedding theorem in Lorentz-Zygmund spaces},
%    journal={Potential Anal.},
%    volume={5},
%    date={1996},
%    number={6},
%    pages={581--590},
%    issn={0926-2601},
%    review={\MR{1437585}},
%    doi={10.1007/BF00275795},
% }


\bib{HS}{article}{
   author={Hernandez, Felipe},
   author={Spector, Daniel},
   title={Fractional integration and optimal estimates for elliptic systems},
   journal={Calc. Var. Partial Differential Equations},
   volume={63},
   date={2024},
   number={5},
   pages={Paper No. 117, 29},
   issn={0944-2669},
   review={\MR{4739434}},
   doi={10.1007/s00526-024-02722-8},
}



\bib{HRS}{article}{
   author={Hernandez, Felipe},
   author={Rai\c{t}\u{a}, Bogdan},
   author={Spector, Daniel},
   title={Endpoint $L^1$ estimates for Hodge systems},
   journal={Math. Ann.},
   volume={385},
   date={2023},
   number={3-4},
   pages={1923--1946},
   issn={0025-5831},
   review={\MR{4566709}},
   doi={10.1007/s00208-022-02383-y},
}


\bib{Hol}{article}{
   author={Holmstedt, Tord},
   title={Interpolation of quasi-normed spaces},
   journal={Math. Scand.},
   volume={26},
   date={1970},
   pages={177--199},
   issn={0025-5521},
   review={\MR{415352}},
   doi={10.7146/math.scand.a-10976},
}
	




\bib{KermanPick}{article}{
   author={Kerman, Ron},
   author={Pick, Lubo\v{s}},
   title={Optimal Sobolev imbeddings},
   journal={Forum Math.},
   volume={18},
   date={2006},
   number={4},
   pages={535--570},
   issn={0933-7741},
   review={\MR{2254384}},
   doi={10.1515/FORUM.2006.028},
}




\bib{LanzaniStein}{article}{
   author={Lanzani, Loredana},
   author={Stein, Elias M.},
   title={A note on div curl inequalities},
   journal={Math. Res. Lett.},
   volume={12},
   date={2005},
   number={1},
   pages={57--61},
   issn={1073-2780},
   review={\MR{2122730}},
   doi={10.4310/MRL.2005.v12.n1.a6},
}




\bib{mazya}{article}{
   author={Maz\cprime ja, V. G.},
   title={Classes of domains and imbedding theorems for function spaces},
   language={Russian},
   journal={Dokl. Akad. Nauk SSSR},
   volume={133},
   pages={527--530},
   issn={0002-3264},
   translation={
      journal={Soviet Math. Dokl.},
      volume={1},
      date={1960},
      pages={882--885},
      issn={0197-6788},
   },
   review={\MR{126152}},
}
		

\bib{mihula}{article}{
   author={Mihula, Zden\v{e}k},
   title={Embeddings of homogeneous Sobolev spaces on the entire space},
   journal={Proc. Roy. Soc. Edinburgh Sect. A},
   volume={151},
   date={2021},
   number={1},
   pages={296--328},
   issn={0308-2105},
   review={\MR{4202643}},
   doi={10.1017/prm.2020.14},
}


% \bib{Mizuta}{book}{
%    author={Mizuta, Yoshihiro},
%    title={Potential theory in Euclidean spaces},
%    series={GAKUTO International Series. Mathematical Sciences and
%    Applications},
%    volume={6},
%    publisher={Gakk$\B$ otosho Co., Ltd., Tokyo},
%    date={1996},
%    pages={viii+341},
%    isbn={4-7625-0415-7},
%    review={\MR{1428685}},
% }

\bib{nirenberg}{article}{
   author={Nirenberg, L.},
   title={On elliptic partial differential equations},
   journal={Ann. Scuola Norm. Sup. Pisa Cl. Sci. (3)},
   volume={13},
   date={1959},
   pages={115--162},
   issn={0391-173X},
   review={\MR{109940}},
}

\bib{oneil}{article}{
   author={O'Neil, Richard},
   title={Convolution operators and {$L(p,\,q)$} spaces},
   journal={Duke Math. J.},
   volume={30},
   date={1963},
   pages={129--142},
   issn={0012-7094},
   %review={\MR{0194881}},
   %doi={10.2307/1994271},
}







 

\bib{peetre}{article}{
   author={Peetre, Jaak},
   title={Espaces d'interpolation et th\'{e}or\`eme de {S}oboleff},
   journal={Ann. Inst. Fourier (Grenoble)},
   volume={16},
   date={1966},
   pages={279--317},
   issn={0373-0956},
   review={\MR{MR221282}},
   doi={ },
}

\bib{Pisier}{article}{
   author={Pisier, Gilles},
   title={Interpolation between $H^p$ spaces and noncommutative
   generalizations. I},
   journal={Pacific J. Math.},
   volume={155},
   date={1992},
   number={2},
   pages={341--368},
   issn={0030-8730},
   review={\MR{1178030}},
}


 
		
\bib{RaitaSpector}{article}{
   author={Rai\c{t}\u{a}, Bogdan},
   author={Spector, Daniel},
   title={A note on estimates for elliptic systems with $L^1$ data},
   journal={C. R. Math. Acad. Sci. Paris},
   volume={357},
   date={2019},
   number={11-12},
   pages={851--857},
   issn={1631-073X},
   review={\MR{4038260}},
   doi={10.1016/j.crma.2019.11.007},
}

\bib{RSS}{article}{
   author={Rai\c{t}\u{a}, Bogdan},
   author={Spector, Daniel},
   author={Stolyarov, Dmitriy},
   title={A trace inequality for solenoidal charges},
   journal={Potential Anal.},
   volume={59},
   date={2023},
   number={4},
   pages={2093--2104},
   issn={0926-2601},
   review={\MR{4684387}},
   doi={10.1007/s11118-022-10008-x},
}



\bib{SSVS}{article}{
   author={Schikorra, Armin},
   author={Spector, Daniel},
   author={Van Schaftingen, Jean},
   title={An $L^1$-type estimate for Riesz potentials},
   journal={Rev. Mat. Iberoam.},
   volume={33},
   date={2017},
   number={1},
   pages={291--303},
   issn={0213-2230},
   review={\MR{3615452}},
   doi={10.4171/RMI/937},
}




\bib{sobolev}{article}{
   author={Sobolev, S.L.},
    title={On a theorem of functional analysis},
   journal={Mat. Sb.},
   volume={4},
   number={46},
  year={1938},
  language={Russian},
   pages={471-497},
   translation={
      journal={Transl. Amer. Math. Soc.},
      volume={34},
     date={},
      pages={39-68},
   },
   }



\bib{Spector-VanSchaftingen-2018}{article}{
   author={Spector, Daniel},
   author={Van Schaftingen, Jean},
   title={Optimal embeddings into Lorentz spaces for some vector
   differential operators via Gagliardo's lemma},
   journal={Atti Accad. Naz. Lincei Rend. Lincei Mat. Appl.},
   volume={30},
   date={2019},
   number={3},
   pages={413--436},
   issn={1120-6330},
   review={\MR{4002205}},
   doi={10.4171/RLM/854},
}


\bib{Stolyarov}{article}{
   author={Stolyarov, D. M.},
   title={Hardy-Littlewood-Sobolev inequality for $p=1$},
   language={Russian, with Russian summary},
   journal={Mat. Sb.},
   volume={213},
   date={2022},
   number={6},
   pages={125--174},
   issn={0368-8666},
   translation={
      journal={Sb. Math.},
      volume={213},
      date={2022},
      number={6},
      pages={844--889},
      issn={1064-5616},
   },
   review={\MR{4461456}},
   doi={10.4213/sm9645},
}

\bib{Strauss}{article}{
   author={Strauss, Monty J.},
   title={Variations of Korn's and Sobolev's equalities},
   conference={
      title={Partial differential equations},
      address={Univ. California,
      Berkeley, Calif.},
      date={1971},
   },
   book={
      publisher={Amer. Math. Soc., Providence, R.I.},
      series={Proc. Sympos. Pure Math.},
      volume={XXIII},
   },
   date={1973},
   pages={207--214},
}


% \bib{Stein}{book}{
%    author={Stein, Elias M.},
%    title={Singular integrals and differentiability properties of functions},
%    series={Princeton Mathematical Series, No. 30},
%    publisher={Princeton University Press, Princeton, N.J.},
%    date={1970},
%    pages={xiv+290},
%    %review={\MR{0290095}},
% }



% \bib{SteinWeiss}{article}{
%    author={Stein, Elias M.},
%    author={Weiss, Guido},
%    title={On the theory of harmonic functions of several variables. I. The
%    theory of $H^{p}$-spaces},
%    journal={Acta Math.},
%    volume={103},
%    date={1960},
%    pages={25--62},
%    issn={0001-5962},
%    %review={\MR{0121579}},
%  %  doi={10.1007/BF02546524},
% }












\bib{VS}{article}{
   author={Van Schaftingen, Jean},
   title={A simple proof of an inequality of Bourgain, Brezis and Mironescu},
   language={English, with English and French summaries},
   journal={C. R. Math. Acad. Sci. Paris},
   volume={338},
   date={2004},
   number={1},
   pages={23--26},
   issn={1631-073X},
   review={\MR{2038078}},
   doi={10.1016/j.crma.2003.10.036},
}


\bib{VS2}{article}{
   author={Van Schaftingen, Jean},
   title={Estimates for $L^1$ vector fields under higher-order differential
   conditions},
   journal={J. Eur. Math. Soc. (JEMS)},
   volume={10},
   date={2008},
   number={4},
   pages={867--882},
   issn={1435-9855},
   review={\MR{2443922}},
   doi={10.4171/JEMS/133},
}

\bib{VS2a}{article}{
   author={Van Schaftingen, Jean},
   title={Limiting fractional and Lorentz space estimates of differential
   forms},
   journal={Proc. Amer. Math. Soc.},
   volume={138},
   date={2010},
   number={1},
   pages={235--240},
   issn={0002-9939},
%    review={\MR{2550188}},
   doi={10.1090/S0002-9939-09-10005-9},
}
\bib{VS3}{article}{
   author={Van Schaftingen, Jean},
   title={Limiting Sobolev inequalities for vector fields and canceling
   linear differential operators},
   journal={J. Eur. Math. Soc. (JEMS)},
   volume={15},
   date={2013},
   number={3},
   pages={877--921},
   issn={1435-9855},
   review={\MR{3085095}},
   doi={10.4171/JEMS/380},
}

\bib{VS4}{article}{
   author={Van Schaftingen, Jean},
   title={Limiting Bourgain-Brezis estimates for systems of linear
   differential equations: theme and variations},
   journal={J. Fixed Point Theory Appl.},
   volume={15},
   date={2014},
   number={2},
   pages={273--297},
   issn={1661-7738},
%    review={\MR{3298002}},
   doi={10.1007/s11784-014-0177-0},
}

% \bib{Zygmund}{article}{
%    author={Zygmund, A.},
%    title={On a theorem of Marcinkiewicz concerning interpolation of
%    operations},
%    journal={J. Math. Pures Appl. (9)},
%    volume={35},
%    date={1956},
%    pages={223--248},
%    %issn={0021-7824},
%    %review={\MR{0080887}},
% }



\end{biblist}
	
\end{bibdiv}
\end{document}